\documentclass[12pt]{amsart}       
\usepackage{txfonts}
\usepackage{amssymb}
\usepackage{eucal}
\usepackage{graphicx}
\usepackage{amsmath}
\usepackage{amscd}
\usepackage[all]{xy}           
\usepackage{tikz}
\usepackage{amsfonts,latexsym}
\usepackage{xspace}
\usepackage{epsfig}
\usepackage{float}
\usepackage{color}
\usepackage{fancybox}
\usepackage{colordvi}
\usepackage{multicol}
\usepackage{colordvi}
 \usepackage{ifpdf}
  \ifpdf
   \usepackage[colorlinks,final,backref=page,hyperindex]{hyperref}
  \else
   \usepackage[colorlinks,final,backref=page,hyperindex,hypertex]{hyperref}
  \fi
\usepackage[active]{srcltx} 
\usepackage{mathrsfs}
\usepackage{graphicx}

\newtheorem{theorem}{Theorem}[section]
\newtheorem{prop}[theorem]{Proposition}
\newtheorem{lemma}[theorem]{Lemma}
\newtheorem{coro}[theorem]{Corollary}
\newtheorem{prop-def}{Proposition-Definition}[section]
\newtheorem{coro-def}{Corollary-Definition}[section]

\theoremstyle{definition}
\newtheorem{defn}[theorem]{Definition}
\newtheorem{remark}[theorem]{Remark}

\newtheorem{algorithm}[theorem]{Algorithm}

\newcommand{\nc}{\newcommand}
\nc{\tred}[1]{\textcolor{red}{#1}}
\nc{\tblue}[1]{\textcolor{blue}{#1}}
\nc{\tgreen}[1]{\textcolor{green}{#1}}
\nc{\tpurple}[1]{\textcolor{purple}{#1}}
\nc{\btred}[1]{\textcolor{red}{\bf #1}}
\nc{\btblue}[1]{\textcolor{blue}{\bf #1}}
\nc{\btgreen}[1]{\textcolor{green}{\bf #1}}
\nc{\btpurple}[1]{\textcolor{purple}{\bf #1}}
\nc{\NN}{{\mathbb N}}
\nc{\ncsha}{{\mbox{\cyr X}^{\mathrm NC}}} \nc{\ncshao}{{\mbox{\cyr
X}^{\mathrm NC}_0}}

\renewcommand{\frak}{\mathfrak}

\newcommand{\efootnote}[1]{}

\renewcommand{\textbf}[1]{}

\newcommand{\delete}[1]{}

\nc{\mlabel}[1]{\label{#1}}  
\nc{\mcite}[1]{\cite{#1}}  
\nc{\mref}[1]{\ref{#1}}  
\nc{\mbibitem}[1]{\bibitem{#1}} 

\delete{
\nc{\mlabel}[1]{\label{#1}{\hfill \hspace{1cm}{\bf{{\ }\hfill(#1)}}}}
\nc{\mcite}[1]{\cite{#1}{{\bf{{\ }(#1)}}}}  
\nc{\mref}[1]{\ref{#1}{{\bf{{\ }(#1)}}}}  
\nc{\mbibitem}[1]{\bibitem[\bf #1]{#1}} 
}

\nc{\tforall}{\quad \text{ for all }}
\nc{\gsb}{Gr\"obner-Shirshov basis\xspace}
\nc{\gsbs}{Gr\"obner-Shirshov bases\xspace}
\nc{\mdl}{\text{dl}}
\nc{\opa}{\ast} \nc{\opb}{\odot} \nc{\op}{\bullet} \nc{\pa}{\frakL}
\nc{\arr}{\rightarrow} \nc{\lu}[1]{(#1)} \nc{\mult}{\mrm{mult}}
\nc{\diff}{\mathfrak{Diff}}
\nc{\opc}{\sharp}\nc{\opd}{\natural}
\nc{\ope}{\circ}
\nc{\dpt}{\mathrm{d}}
\nc{\hck}{H_{RT}}
\nc{\vdf}{\calf}
\nc{\ldf}{\calf_\ell}
\nc{\hlf}{H_\ell}
\nc{\onek}{\mathbf{1}_\bfk}

\nc{\mrba}{matching Rota-Baxter algebra\xspace}
\nc{\Mrba}{Matching Rota-Baxter algebra\xspace}
\nc{\mrbas}{matching Rota-Baxter algebras\xspace}
\nc{\Mrbas}{Matching Rota-Baxter algebras\xspace}

\nc{\match}{matching\xspace}
\nc{\Match}{Matching\xspace}

\nc{\paybe}{polarized associative Yang-Baxter equation\xspace}
\nc{\Paybe}{Polarized associative Yang-Baxter equation\xspace}
\nc{\cpaybe}{PAYBE}

\nc{\rba}{Rota-Baxter algebra\xspace}
\nc{\rbas}{Rota-Baxter algebras\xspace}

\nc{\diam}{alternating\xspace}
\nc{\Diam}{Alternating\xspace}
\nc{\cdiam}{canonical alternating\xspace}
\nc{\Cdiam}{Canonical alternating\xspace}
\nc{\AW}{\mathcal{A}}

\nc{\ari}{\mathrm{ar}}

\nc{\lef}{\mathrm{lef}}

\nc{\Sh}{\mathrm{ST}}

\nc{\Cr}{\mathrm{Cr}}

\nc{\st}{{Schr\"oder tree}\xspace}
\nc{\sts}{{Schr\"oder trees}\xspace}

\nc{\vertset}{\Omega} 

\nc{\assop}{\quad \begin{picture}(5,5)(0,0)
\line(-1,1){10}
\put(-2.2,-2.2){$\bullet$}
\line(0,-1){10}\line(1,1){10}
\end{picture} \quad \smallskip}

\nc{\operator}{\begin{picture}(5,5)(0,0)
\line(0,-1){6}
\put(-2.6,-1.8){$\bullet$}
\line(0,1){9}
\end{picture}}

\nc{\idx}{\begin{picture}(6,6)(-3,-3)
\put(0,0){\line(0,1){6}}
\put(0,0){\line(0,-1){6}}
\end{picture}}

\nc{\pb}{{\mathrm{pb}}}
\nc{\Lf}{{\mathrm{Lf}}}

\nc{\lft}{{left tree}\xspace}
\nc{\lfts}{{left trees}\xspace}

\nc{\fat}{{fundamental averaging tree}\xspace}

\nc{\fats}{{fundamental averaging trees}\xspace}
\nc{\avt}{\mathrm{Avt}}

\nc{\rass}{{\mathit{RAss}}}

\nc{\aass}{{\mathit{AAss}}}

\nc{\vin}{{\mathrm Vin}}    
\nc{\lin}{{\mathrm Lin}}    
\nc{\inv}{\mathrm{I}n}
\nc{\gensp}{V} 
\nc{\genbas}{\mathcal{V}} 
\nc{\bvp}{V_P}     
\nc{\gop}{{\,\omega\,}}     

\nc{\bin}[2]{ (_{\stackrel{\scs{#1}}{\scs{#2}}})}  
\nc{\binc}[2]{ \left (\!\! \begin{array}{c} \scs{#1}\\
    \scs{#2} \end{array}\!\! \right )}  
\nc{\bincc}[2]{  \left ( {\scs{#1} \atop
    \vspace{-1cm}\scs{#2}} \right )}  
\nc{\bs}{\bar{S}} \nc{\cosum}{\sqsubset} \nc{\la}{\longrightarrow}
\nc{\rar}{\rightarrow} \nc{\dar}{\downarrow} \nc{\dprod}{**}
\nc{\dap}[1]{\downarrow \rlap{$\scriptstyle{#1}$}}
\nc{\md}{\mathrm{dth}} \nc{\uap}[1]{\uparrow
\rlap{$\scriptstyle{#1}$}} \nc{\defeq}{\stackrel{\rm def}{=}}
\nc{\disp}[1]{\displaystyle{#1}} \nc{\dotcup}{\
\displaystyle{\bigcup^\bullet}\ } \nc{\gzeta}{\bar{\zeta}}
\nc{\hcm}{\ \hat{,}\ } \nc{\hts}{\hat{\otimes}}
\nc{\barot}{{\otimes}} \nc{\free}[1]{\bar{#1}}
\nc{\uni}[1]{\tilde{#1}} \nc{\hcirc}{\hat{\circ}} \nc{\lleft}{[}
\nc{\lright}{]} \nc{\lc}{\lfloor} \nc{\rc}{\rfloor}
\nc{\curlyl}{\left \{ \begin{array}{c} {} \\ {} \end{array}
    \right .  \!\!\!\!\!\!\!}
\nc{\curlyr}{ \!\!\!\!\!\!\!
    \left . \begin{array}{c} {} \\ {} \end{array}
    \right \} }
\nc{\longmid}{\left | \begin{array}{c} {} \\ {} \end{array}
    \right . \!\!\!\!\!\!\!}
\nc{\onetree}{\bullet} \nc{\ora}[1]{\stackrel{#1}{\rar}}
\nc{\ola}[1]{\stackrel{#1}{\la}}
\nc{\ot}{\otimes} \nc{\mot}{{{\boxtimes\,}}}
\nc{\otm}{\overline{\boxtimes}} \nc{\sprod}{\bullet}
\nc{\scs}[1]{\scriptstyle{#1}} \nc{\mrm}[1]{{\rm #1}}
\nc{\margin}[1]{\marginpar{\rm #1}}   
\nc{\dirlim}{\displaystyle{\lim_{\longrightarrow}}\,}
\nc{\invlim}{\displaystyle{\lim_{\longleftarrow}}\,}
\nc{\mvp}{\vspace{0.3cm}} \nc{\tk}{^{(k)}} \nc{\tp}{^\prime}
\nc{\ttp}{^{\prime\prime}} \nc{\svp}{\vspace{2cm}}
\nc{\vp}{\vspace{8cm}} \nc{\proofbegin}{\noindent{\bf Proof: }}
\nc{\proofend}{$\blacksquare$ \vspace{0.3cm}}
\nc{\modg}[1]{\!<\!\!{#1}\!\!>}
\nc{\intg}[1]{F_C(#1)} \nc{\lmodg}{\!
<\!\!} \nc{\rmodg}{\!\!>\!}
\nc{\cpi}{\widehat{\Pi}}
\nc{\sha}{{\mbox{\cyr X}}}  

\newfont{\scyr}{wncyr10 scaled 550}
\nc{\ssha}{\mbox{\bf \scyr X}}
\nc{\shap}{{\mbox{\cyrs X}}} 
\nc{\shpr}{\diamond}    
\nc{\shp}{\ast} \nc{\shplus}{\shpr^+}
\nc{\shprc}{\shpr_c}    
\nc{\msh}{\ast} \nc{\zprod}{m_0} \nc{\oprod}{m_1}
\nc{\vep}{\epsilon} \nc{\labs}{\mid\!} \nc{\rabs}{\!\mid}
\nc{\sqmon}[1]{\langle #1\rangle}

\nc{\mmbox}[1]{\mbox{\ #1\ }} \nc{\dep}{\mrm{dep}} \nc{\fp}{\mrm{FP}}
\nc{\rchar}{\mrm{char}} \nc{\End}{\mrm{End}} \nc{\Fil}{\mrm{Fil}}
\nc{\Mor}{Mor\xspace} \nc{\gmzvs}{gMZV\xspace}
\nc{\gmzv}{gMZV\xspace} \nc{\mzv}{MZV\xspace}
\nc{\mzvs}{MZVs\xspace} \nc{\Hom}{\mrm{Hom}} \nc{\id}{\mrm{id}}
\nc{\im}{\mrm{im}} \nc{\incl}{\mrm{incl}} \nc{\map}{\mrm{Map}}
\nc{\mchar}{\rm char} \nc{\nz}{\rm NZ} \nc{\supp}{\mathrm Supp}

\nc{\Alg}{\mathbf{Alg}} \nc{\Bax}{\mathbf{Bax}} \nc{\bff}{\mathbf f}
\nc{\bfk}{{\bf k}} \nc{\bfone}{{\bf 1}} \nc{\bfx}{\mathbf x}
\nc{\bfy}{\mathbf y}
\nc{\base}[1]{\bfone^{\otimes ({#1}+1)}} 
\nc{\Cat}{\mathbf{Cat}}

\nc{\detail}{\marginpar{\bf More detail}
    \noindent{\bf Need more detail!}
    \svp}
\nc{\Int}{\mathbf{Int}} \nc{\Mon}{\mathbf{Mon}}
\nc{\rbtm}{{shuffle }} \nc{\rbto}{{Rota-Baxter }}
\nc{\remarks}{\noindent{\bf Remarks: }} \nc{\Rings}{\mathbf{Rings}}
\nc{\Sets}{\mathbf{Sets}} \nc{\wtot}{\widetilde{\odot}}
\nc{\wast}{\widetilde{\ast}} \nc{\bodot}{\bar{\odot}}
\nc{\bast}{\bar{\ast}} \nc{\hodot}[1]{\odot^{#1}}
\nc{\hast}[1]{\ast^{#1}} \nc{\mal}{\mathcal{O}}
\nc{\tet}{\tilde{\ast}} \nc{\teot}{\tilde{\odot}}
\nc{\oex}{\overline{x}} \nc{\oey}{\overline{y}}
\nc{\oez}{\overline{z}} \nc{\oef}{\overline{f}}
\nc{\oea}{\overline{a}} \nc{\oeb}{\overline{b}}
\nc{\weast}[1]{\widetilde{\ast}^{#1}}
\nc{\weodot}[1]{\widetilde{\odot}^{#1}} \nc{\hstar}[1]{\star^{#1}}
\nc{\lae}{\langle} \nc{\rae}{\rangle}
\nc{\lf}{\lfloor}
\nc{\rf}{\rfloor}

\nc{\QQ}{{\mathbb Q}}
\nc{\RR}{{\mathbb R}} \nc{\ZZ}{{\mathbb Z}}

\nc{\cala}{{\mathcal A}} \nc{\calb}{{\mathcal B}}
\nc{\calc}{{\mathcal C}}
\nc{\cald}{{\mathcal D}} \nc{\cale}{{\mathcal E}}
\nc{\calf}{{\mathcal F}} \nc{\calg}{{\mathcal G}}
\nc{\calh}{{\mathcal H}} \nc{\cali}{{\mathcal I}}
\nc{\call}{{\mathcal L}} \nc{\calm}{{\mathcal M}}
\nc{\caln}{{\mathcal N}} \nc{\calo}{{\mathcal O}}
\nc{\calp}{{\mathcal P}} \nc{\calr}{{\mathcal R}}
\nc{\cals}{{\mathcal S}} \nc{\calt}{{\mathcal T}}
\nc{\calu}{{\mathcal U}} \nc{\calw}{{\mathcal W}} \nc{\calk}{{\mathcal K}}
\nc{\calx}{{\mathcal X}} \nc{\CA}{\mathcal{A}}

\nc{\fraka}{{\mathfrak a}} \nc{\frakA}{{\mathfrak A}}
\nc{\frakb}{{\mathfrak b}} \nc{\frakB}{{\mathfrak B}}
\nc{\frakc}{{\mathfrak c}}
\nc{\frakD}{{\mathfrak D}} \nc{\frakF}{\mathfrak{F}}
\nc{\frakf}{{\mathfrak f}} \nc{\frakg}{{\mathfrak g}}
\nc{\frakH}{{\mathfrak H}} \nc{\frakL}{{\mathfrak L}}
\nc{\frakM}{{\mathfrak M}} \nc{\bfrakM}{\overline{\frakM}}
\nc{\frakm}{{\mathfrak m}} \nc{\frakP}{{\mathfrak P}}
\nc{\frakN}{{\mathfrak N}} \nc{\frakp}{{\mathfrak p}}
\nc{\frakS}{{\mathfrak S}} \nc{\frakT}{\mathfrak{T}}
\nc{\frakX}{{\mathfrak X}}
\nc{\BS}{\mathbb{S
}}

\font\cyr=wncyr10 \font\cyrs=wncyr7

\nc{\ID}{{\rm I}}\nc{\lbar}[1]{\overline{#1}}\nc{\bre}{{\rm bre}}
\nc{\sd}{\cals}\nc{\rb}{\rm RB}\nc{\A}{\rm A}\nc{\LL}{\rm L}\nc{\tx}{\tilde{X}}
\nc{\col}{\Delta_{RT}}\nc{\mul}{m_{RT}}\nc{\ul}{u_{RT}}\nc{\epl}{\varepsilon_{RT}}
\nc{\hl}{H_{RT}}\nc{\arro}[1]{#1}\nc{\px}{P_{\tx}}\nc{\pw}{P_{\mathfrak{w}}}\nc{\pl}{B_\omega^+}
\nc{\pp}{\pl}\nc{\ppp}[1]{B^+(#1)}\nc{\dw}{\diamond_{\mathfrak{w}}}\nc{\dl}{\diamond_{\rm \ell}}
\nc{\ncshaw}{\sha^{{\rm NC}}_{\Omega}}\nc{\ncshal}{\sha^{{\rm NC}}_{{\rm RT}}}
\nc{\ver}{\rm V}\nc{\ld}{l}\nc{\del}{\Delta_{{\rm \ell}}}\nc{\epsl}{\epsilon_{{\rm \ell}}}
\nc{\uul}{u_{{\rm \ell}}}\nc{\oneh}{\mathbf{1}}\nc{\onew}{\mathbf{1}}
\nc{\etree}{1} \nc{\conc}{m_{RT}} \nc{\mpu}{u^{\ast}} \nc{\mpv}{v^{\ast}}
\nc{\brep}{\text{bre}_{P}} \nc{\leqo}{\leq_{\text{db}}} \nc{\odb}{<_{\text{db}}}

\nc{\hrtb}{H_{RT}(X\sqcup\Omega)} \nc{\hrts}{H_{RT}(X, \Omega)}\nc{\rts}{\mathcal{T}(X, \Omega)}\nc{\rfs}{\mathcal{F}(X, \Omega)} \nc{\ncshall}{\sha^{{\rm NC}}_{{\rm RT}}} \nc{\ldl}{\leq_{\mathrm{db}}} \nc{\pla}{B_{\alpha}^{+}} \nc{\plb}{B_{\beta}^{+}}

\nc{\bim}[1]{#1}  \nc{\shaop}{\sha_{\Omega}^{+}}  \nc{\shao}{\sha_{\Omega}}
\nc{\bbim}[2]{#1 #2} \nc{\bbbim}[2]{#1,\, #2} \nc{\RBF}{{\rm RBF}}
\nc{\frbf}{F_{\RBF}} \nc{\shaf}{\ssha_{\tiny{\Omega}}} \nc{\sham}{\diamond_{\Omega}}

\nc{\dnx}{\Delta_n A} \nc{\dx}{\Delta A} \nc{\dgp}{{\rm deg_{P}}}
\nc{\dgt}{{\rm deg_{T}}} \nc{\dg}{{\rm deg}} \nc{\ida}{ID($A$)} \nc{\tu}{\tilde{u}} \nc{\tv}{\tilde{v}}
 \nc{\fbase}{\calb} \nc{\LF}{\mathrm{RF}} \nc{\FFA}{\mathrm{LF}} \nc{\irr}{\mathrm{Irr}}
 \nc{\result}{\bfk\mathrm{Irr}(S_n)}  \nc{\I}{I_{\mathrm{ID},n}^0}
 \nc{\nrs}{\calr_n^\star} \nc{\ii}{\mathrm{I}} \nc{\iii}{\mathrm{II}}
\nc{\intl}{{\rm int}}\nc{\ws}[1]{{#1}}\nc{\deleted}[1]{\delete{#1}}\nc{\plas}{placements\xspace}

\nc{\Id}{\mathrm{Id}} \nc{\Irr}{\mathrm{Irr}}

\nc{\tos}{totally ordered set }
\nc{\nes}{nonempty set}
\nc{\Po}{(P_\omega)_{\omega\in \Omega}}

\nc{\Pop}{(P'_\omega)_{\omega\in \Omega}}

\nc{\Bo}{(B_{\omega}^+)_{\omega\in \Omega}}

\newcommand{\tdun}[1]
{\begin{picture}(10,5)(-2,-1)
\put(0,0){\circle*{2}}
\put(3,-2){\tiny #1}
\end{picture}}

\newcommand{\tddeux}[2]{\begin{picture}(12,5)(0,-1)
\put(3,0){\circle*{2}}
\put(3,0){\line(0,1){5}}
\put(3,5){\circle*{2}}
\put(6,-3){\tiny #1}
\put(6,3){\tiny #2}
\end{picture}}

\newcommand{\tdtroisun}[3]{\begin{picture}(20,12)(-5,-1)
\put(3,0){\circle*{2}}
\put(-0.65,0){$\vee$}
\put(6,7){\circle*{2}}
\put(0,7){\circle*{2}}
\put(5,-2){\tiny #1}
\put(8,5){\tiny #2}
\put(-6,5){\tiny #3}
\end{picture}}

\newcommand{\tdquatretrois}[4]{\begin{picture}(20,20)(-5,-1)
\put(3,0){\circle*{2}}
\put(-.65,0){$\vee$}
\put(6,7){\circle*{2}}
\put(0,7){\circle*{2}}
\put(6,14){\circle*{2}}
\put(6,7){\line(0,1){7}}
\put(5,-2){\tiny #1}
\put(8,5){\tiny #2}
\put(-6,5){\tiny #4}
\put(8,12){\tiny #3}
\end{picture}}

\definecolor{red}{rgb}{1.,0.,0.}
\definecolor{green}{rgb}{0.,1.,0.}
\definecolor{blue}{rgb}{0.,0.,1.}

\begin{document}

\title[Hopf algebras of rooted forests and free matching Rota-Baxter algebras]{Hopf algebra of multi-decorated rooted forests, free matching Rota-Baxter algebras and Gr\"obner-Shirshov bases}
%
\author{Xing Gao}
\address{School of Mathematics and Statistics,
Key Laboratory of Applied Mathematics and Complex Systems,
Lanzhou University, Lanzhou, Gansu 730000, P.\,R. China}
\email{gaoxing@lzu.edu.cn}

\author{Li Guo}
\address{Department of Mathematics and Computer Science, Rutgers University, Newark, NJ 07102, USA}
\email{liguo@rutgers.edu}

\author{Yi Zhang}
\address{School of Mathematics and Statistics,
Nanjing University of Information Science \& Technology, Nanjing, Jiangsu 210044, P.\,R. China}
\email{zhangy2016@nuist.edu.cn}


\date{\today}
\begin{abstract}

Recent advances in stochastic PDEs, Hopf algebras of typed trees and integral equations have inspired the study of algebraic structures with replicating operations. To understand their algebraic and combinatorial nature, we first use rooted forests with multiple decoration sets to construct free Hopf algebras with multiple Hochschild 1-cocycle conditions. Applying the universal property of the underlying operated algebras and the method of \gsbs, we then construct free objects in the category of matching Rota-Baxter algebras which is a generalization of Rota-Baxter algebras to allow multiple Rota-Baxter operators. Finally the free matching Rota-Baxter algebras are equipped with a cocycle Hopf algebra structure.
\end{abstract}

\subjclass[2010]{
16T30,	
17B38,   
05E16,   
16Z10, 
16T10, 
05C05,   
16W99, 
16S10, 
05A05   
}

\keywords{Rooted forest; Hopf algebra; Rota-Baxter algebra; 1-cocycle condition; Gr\"{o}bner-Shirshov bases}

\maketitle

\vspace{-1cm}

\tableofcontents

\vspace{-1cm}
\setcounter{section}{0}

\allowdisplaybreaks

\section{Introduction}

This paper applies rooted forests to obtain free objects for  generalized cocycle Hopf algebras and then for matching Rota-Baxter algebras. To pass from the former to the latter, we interpret rooted forests as bracketed words and utilize the method of Gr\"obner-Shirshov bases.

\subsection{Rooted tree Hopf algebras and Rota-Baxter algebras}

The Hopf algebra of rooted forests arose from the study of Connes and Kreimer~\mcite{CK98,Kr98} on renormalization of quantum field theory where the rooted forests serves as a baby model of Feynman diagrams. The importance of this Hopf algebra and its noncommutative analog, the Foissy-Holtkamp Hopf algebra~\mcite{Foi02,Hol03}, has been investigated from various points of view. From the algebraic viewpoint, this importance is revealed by its universal property in the category of cocycle Hopf algebras~\mcite{Foi02,KP,Moe01}, in terms of a Hopf algebra with a grafting operator tied together by the Hochschild 1-cocycle condition. It is also characterized as the free operated algebra~\mcite{Guo09,ZGG16}, namely an algebra equipped with a linear operator, and thus are interpreted in terms of bracketed words and Motzkin paths.
Recently such a universal property was generalized to braided Hopf algebras of rooted forests \`a la Connes-Kreimer~\mcite{Foi08,GL19}.

Another algebraic structure playing a key role in the work of Connes and Kreimer, and in the context of operated algebras, is the Rota-Baxter algebra which has its origin in the work of G. Baxter~\mcite{Bax60} in fluctuation theory in probability. Connections of Rota-Baxter algebras have been established with broad areas in mathematics and mathematical physics, including quasi-symmetric functions, operads, integrable system and renormalization methods. See for example~\mcite{Ag,BGN,GPZ,YGT}.

Free Rota-Baxter algebras can also be realized on rooted forests~\mcite{EG08,ZGG16}. In fact, the realization can be obtained from certain generalization of the noncommutative Connes-Kreimer Hopf algebra. This connection of free Rota-Baxter algebras with generalized Connes-Kreimer Hopf algebras not only highlights the combinatorial nature of Rota-Baxter algebras which attracted the attentions of outstanding combinatorists such as Cartier and Rota~\mcite{Ca,Ro}, but also provides a natural Hopf algebra structure on free Rota-Baxter algebras from that on the rooted forests.

Further various Hochschild 1-cocycle conditions have been applied to establish or characterize other Hopf like algebraic structures, such as Hopf algebras on free commutative modified Rota-Baxter algebras~\mcite{ZGG19}, left counital Hopf algebras on  free (commutative) Nijenhuis algebras~\mcite{GLZ18, ZG17} and free Rota-Baxter systems~\mcite{PZGL, QC18}, as well as the Loday-Ronco Hopf algebra of binary rooted trees~\mcite{ZG18} and infinitesimal bialgebras of rooted forests~\mcite{ZCGL18}.

\subsection{Matching Rota-Baxter algebras and outline of the paper}

As a multi-operator generalization of the Rota-Baxter algebra, the recent notion of a \mrba~\cite{GGZ19} has its motivation from the study of multiple pre-Lie algebras~\mcite{Foi18} originated in the important work of Bruned, Hairer and Zambotti~\mcite{BHZ} on algebraic renormalization of regularity structures and further motivated by the studies of associative Yang-Baxter equations, Volterra integral equations and linear structure of Rota-Baxter operators~\mcite{GGZy,GGL}. See Section~\mref{ss:mrba} for a summary of the broad connections of \mrbas.

Our purpose of this paper is to construct free \mrbas from rooted forests by generalizing the construction of free Rota-Baxter algebras mentioned above from the cocycle Hopf algebra of rooted forests, from one operator to multiple operators. Thus we first introduce a class of decorated rooted forests which will serve as the carrier of the free Hopf algebra with multiple Hochschild 1-cocycle conditions. Free \match \rbas will be a quotient of this free cocycle Hopf algebra modulo the operated ideal generated by the \match \rba relations. We next display a basis of this quotient for which we apply the method of \gsbs. As this method works better with the algebraic notion of bracketed words, we utilize the dictionary between rooted forests and bracketed words provided in~\mcite{Guo09} which allows us to take advantage of both the combinatorial structure of rooted forests for their ease description of the coproduct and the algebraic structure of bracketed words for their amenability for detailed computations.

In the process, we also take advantage of the two gradings and filtrations  on rooted forests (and the corresponding bracketed words), one by the number of vertices and one by the depths of the forests.

To provide more details, our first step is to construct free multiple cocycle Hopf algebras from rooted forests Section~\mref{sec:oophopf}.
Thus we will need to work with a set $X$ of generators and a set $\Omega$ of operators that satisfy the cocycle condition. For this purpose, we introduce rooted forests with two decoration sets $X$ and $\Omega$, with $X$ only allowed to decorate the leaf vertices. We then show that the resulting space $\hrts$ meets our needs for the free $\Omega$-cocycle Hopf algebra on $X$ (Theorem~\mref{thm:propm}). Then the free \match \rba on $X$ is the quotient of $\hrts$ modulo the \match \rba relation. In order to apply the method of \gsbs to give an explicit construction of the free \match \rba in the next section, we utilize the one-one correspondence between rooted forests and bracketed words~\mcite{Guo09} to rephrase Theorem~\mref{thm:propm} in terms of (multiple) bracketed words (Corollary~\mref{co:wfil}).

Our goal in Section~\mref{sec:freemrb} is to give an explicit construction of free matching Rota-Baxter algebras by applying the method of Gr\"obner-Shirshov bases to obtain a canonical linear basis of the quotient of the free $\Omega$-cocycle Hopf algebra $\hrts$ (reinterpreted in terms of bracketed words) modulo the \match \rba relations. First the notion of \match \rbas is recalled together with a list of their properties. Then with a suitable monomial order, it is established in Theorem~\mref{thm:GS} that the \match \rba relations form a \gsb and thus give rise to a linear basis of the free \match \rba as the aforementioned quotient. The operations of the \match \rba are given in terms of this linear basis. In view of establishing a cocycle Hopf algebra structure on the free \match \rba in the next section, we apply the one-one correspondence between rooted forests and bracketed words again and rephrase the free \match \rba in terms of decorated rooted forests.

In Section~\mref{sec:hopf}, we equip the free \mrba with an $\Omega$-cocycle Hopf algebra structure descending from the one on $\hrts$, by first establishing an $\Omega$-cocycle bialgebra structure and then verifying the needed connected condition in order to obtain the Hopf algebra structure in Theorem~\mref{thm:cohopf}.

\smallskip

\noindent
{\bf Notations.}
Throughout this paper, we fix a unitary commutative ring $\bfk$
which will be the base ring of all modules, algebras, coalgebras, bialgebras, tensor products, as well as linear maps. By an algebra, we mean a unitary associative algebra unless otherwise specified.

\section{Multi-operated Hopf algebras of decorated rooted forests}
\mlabel{sec:oophopf}
In this section, we construct free operated algebras with multiple operators by decorated rooted forests, as well as by bracketed words.

The rooted forests we consider have different decorations on their leafs and (internal) vertices, but can still be obtained as a suitable subset of the classical noncommutative Connes-Kreimer Hopf algebra $\hck(\Omega)$, that is, Foissy-Holtkamp Hopf algebra~\mcite{Foi02,Hol03}, of rooted forests with all their vertices (leafs and internal vertices) decorated by the same set. Thus we first recall the notions for $\hck(\Omega)$ for later applications.

\subsection{Noncommutative Connes-Kreimer Hopf algebras}
\mlabel{ss:ck}
Various Hopf algebras of decorated planar rooted trees and forests are commonly studied in combinatorics, algebra and other fields. We recall the needed notions and results to be applied in our constructions in this paper.

A $\mathbf{rooted\ tree}$ is a connected and simply connected set of vertices and oriented edges such that there is precisely one distinguished vertex, called the $\mathbf{root}$. A $\mathbf{planar\ rooted\ tree}$  is a rooted tree with a fixed embedding into the plane.

Let $\calt$ denote the set of planar rooted trees and $\calf$ the set of {\bf planar rooted forests}, expressed algebraically as the free monoid $\calf:=M(\calt)$ generated by $\calt$ with the concatenation product $\mul$ which is usually suppressed for brevity. The {\bf empty tree} in $\calf$ is denoted by $1$, the unit of $M(\calt)$.
Thus a planar rooted forest is a noncommutative concatenation of planar rooted trees, denoted by $F=T_1\cdots T_n$ with $T_1, \ldots, T_n \in \calt$, with the convention that $F=1$ when $n=0$.

Given a nonempty set $\Omega$, let $\calt(\Omega)$ (resp. $\calf(\Omega):=M(\calt(\Omega))$) denote the set of planar rooted trees (resp. forests) whose vertices, including the leaves and internal vertices, are decorated by elements of $\Omega$. Define the free $\bfk$-module spanned by the set $\vdf(\Omega)$:
\begin{align*}
\hck(\Omega):= \bfk \vdf(\Omega)=\bfk M(\calt(\Omega)) =\bfk\langle \calt(\Omega)\rangle,
\end{align*}
which is also the noncommutative polynomial algebra on the set $\calt(\Omega)$ with the concatenation.

The noncommutative Connes-Kreimer Hopf algebra $\hck(\Omega)$ introduced by Foissy~\mcite{Foi02} and Holtkamp~\mcite{Hol03} is the above algebra equipped with a coproduct which can be defined in several ways, by subforests, by admissible cuts and by a 1-cocycle condition. A {\bf subforest} of a planar rooted forest $F\in\calf(\Omega)$ is the forest
consisting of a set of vertices of $F$ together with their descendants and edges connecting all these vertices.
Let $\calf_{F}$ be the set of subforests of $F$, including the empty tree 1 and the full subforest $F$.
Define
\begin{equation}
\Delta_{RT} (F):=\sum_{G\in \calf_{F}}G\otimes(F/G),
\mlabel{eq:cop1}
\end{equation}
where $F/G$ is obtained by removing the subforest $G$ and edges connecting $G$ to the rest of the tree~\mcite{Gub}. Here we use the convention that $F/G=1$ when $F=G$, and $F/G=F$ when $G=1$.
See~\mcite{CK98,Foi02} for a description of the coproduct by admissible cuts.

The 1-cocycle condition that characterizes the coproduct $\Delta_{RT}$ is given by the {\bf grafting operators}. For $\omega\in \Omega$, define
$$B^+_\omega:\hck(\Omega)\to \hck(\Omega)$$
to be the linear grafting operation by sending a rooted forest in $\calf(\Omega)$ to its grafting with the new root decorated by $\omega$ and sending $1$ to $\bullet_\omega$.

Then a recursive description of
$\Delta_{RT}$ is the {\bf 1-cocycle condition} for $T\in \calt$
\begin{equation}
\Delta_{RT} B^+_\omega (T):= B^+_\omega(T) \otimes 1+ (\id\otimes B^+_\omega)\col(T)
\mlabel{eq:cocycle}
\end{equation}
with the convention $\Delta_{RT} (1)=1\ot 1$. In particular, we have
\begin{align}
\Delta_{RT} (\bullet_\omega)=\bullet_\omega \ot 1+1\ot \bullet_\omega,\,  \omega\in \Omega. \mlabel{eq:pri}
\end{align}
For a decorated rooted forest $F=T_1\cdots T_m\in \calf(\Omega)$ with $m\geq 2$, we have
\begin{equation}
\Delta_{RT} (F)=\col (T_{1})\cdots\col (T_{m}).
\mlabel{eq:cop2}
\end{equation}

Also define $\epl:\bfk\vdf(\Omega) \to \bfk$ by taking $\epl(F) = 0$ for $F\in \calt(\Omega)$,
and $\epl(1) = 1$.
Let $\ul: {\bfk}\rightarrow \hck(\Omega)$  be the linear map given by $1_{\bfk}\mapsto 1$.

Recall~\mcite{GG19, Gub,Man01} that a bialgebra $(H,m,u,\Delta,\varepsilon)$ is called {\bf graded} if there are {\bfk}-submodules $H^{(n)}, n\geq0$, of $H$ such that
$$H=\bigoplus\limits^{\infty}_{n\geq0}H^{(n)}, \quad H^{(p)}H^{(q)}\subseteq H^{(p+q)}, \quad \Delta(H^{(n)})\subseteq\bigoplus\limits^{}_{p+q=n}H^{(p)}\otimes H^{(q)}, \quad  n, p, q\geq0.$$
The bialgebra $H$ is called {\bf connected graded} if in addition $H^{(0)}=\mathrm{im} u~(={\bfk})$ and $\mathrm{ker}~ \varepsilon=\bigoplus_{n\geq 1}H^{(n)}$.
It is well known that a connected graded bialgebra is a Hopf algebra.

\begin{theorem}\mcite{Foi02, Hol03}
With the degree of a rooted forest defined by its number of vertices, the quintuple $(\hck(\Omega), \mul, \ul, \Delta_{RT} , \epl )$ is a connected graded bialgebra and hence a Hopf algebra.
\mlabel{thm:rt}
\end{theorem}

\subsection{Multi-decorated planar rooted trees and forests}

We now give a generalization of the Hopf algebra $H_{RT}(\Omega)$ of decorated rooted forests by allowing the leaf vertices and internal vertices decorated by different sets. We then show that it gives a realization of the free object in the category of algebras with multiple operators, thus automatically equipping the free object with a Hopf algebra structure given by a cocycle condition.

Let $X$ be a set and let $\Omega$ be a nonempty set disjoint from $X$. Replacing $\Omega$ by $X \sqcup \Omega $ in $\hck(\Omega)$, we obtain the Hopf algebra $\hrtb=\bfk \calf(X\sqcup \Omega)$ as in Theorem~\mref{thm:rt}.

Let $\rts$ (resp.~$\rfs$) denote the subset of $\calt(X\sqcup \Omega)$ (resp.~$\calf(X\sqcup \Omega)$) consisting of vertex decorated planar rooted trees (resp. forests) with the property that elements of only $\Omega$ can  decorate the internal vertices, namely vertices which are not leafs. The unique vertex of the tree $\bullet$ is regarded as a leaf vertex. In other words, elements of $X$ can only be used to decorate the leaf vertices. Of course, some of the leaf vertices can also be decorated by elements from $\Omega$.
For example,
$$\etree, \tdun{$\alpha$},\ \, \tdun{$x$},\ \, \tddeux{$\alpha$}{$\beta$},\ \,  \tddeux{$\alpha$}{$x$}, \ \, \tdtroisun{$\alpha$}{$\beta$}{$\gamma$},\ \,\tdtroisun{$\alpha$}{$x$}{$\gamma$}, \ \,\tdtroisun{$\alpha$}{$x$}{$y$}, \ \, \tdquatretrois{$\alpha$}{$\beta$}{$\gamma$}{$\beta$},\ \, \tdquatretrois{$\alpha$}{$\beta$}{$\gamma$}{$x$}, \ \, \tdquatretrois{$\alpha$}{$\beta$}{$x$}{$y$},\quad x,y,\in X, \alpha,\beta,\gamma\in \Omega,$$
are in $\rts$ whereas, the following are not in $\rts$:
$$\tddeux{$x$}{$\alpha$},\ \,  \tddeux{$x$}{$y$}, \ \, \tdtroisun{$x$}{$\beta$}{$\alpha$}, \ \, \tdquatretrois{$\alpha$}{$x$}{$\gamma$}{$\beta$}, \quad x, y \in X, \alpha,\beta,\gamma\in \Omega. $$

\begin{remark}
Now we give some special cases of our decorated planar rooted forests.
\begin{enumerate}
\item If $X = \emptyset$, then $\mathcal{F}(X, \Omega)=\calf(\Omega)$ is the linear basis in the decorated Foissy-Holtkamp Hopf algebra $\hck(\Omega)$~\mcite{Foi02}. \mlabel{it:2ex}

\item If $\Omega$ is a singleton, then $\rfs$ was introduced and studied in~\mcite{ZGG16} to construct a cocycle Hopf algebra on decorated planar rooted forests.
\item The subset of $\rfs$ consisting of rooted forests whose (all) leaf vertices are decorated by elements of $X$ and whose internal vertices are decorated by elements of $\Omega$, are introduced in~\mcite{Guo09} to construct free operated {\em nonunitary} semigroups and free operated nonunitary algebras.
\end{enumerate}
\mlabel{re:3ex}
\end{remark}

Define
\begin{align*}
\hrts:= \bfk \rfs=\bfk M(\rts)
\end{align*}
to be the free $\bfk$-module spanned by $\rfs$.

We define the degree $\deg(F)$ of $F\in \calf(X,\Omega)$ to be its number of vertices. For $n\geq 0$, let $\calf^{(n)}$ denote the set of $F\in \calf(X,\Omega)$ with degree $n$ and let $\hrts^{(n)}:=H_{RT}^{(n)}:=\bfk \calf^{(n)}$.
For $F=T_1\cdots T_{k}\in \rfs$ with $k\geq 0$ and $T_1,\cdots,T_{k}\in \rts$, define $\bre(F):=k$ to be the {\bf breadth} of $F$ with the convention that $\bre(\etree) = 0$ when $k=0$.

\begin{theorem}
The quintuple $(\hrts,\,\mul,\,1,\,\Delta_{RT},\,\epl)$ is a connected graded subbialgebra of $\hrtb$ with the grading $\hrts=\oplus_{n\geq0}H_{RT}^{(n)}$ and hence is a Hopf algebra.
\mlabel{thm:hopfx}
\end{theorem}
\begin{proof}
Note that $\hrts$ is a subspace of $\hrtb$. It is sufficient to show that $\hrts$ is closed under the multiplication $\mul$ and the coproduct $\Delta_{RT}$. Recall that forests in $\hrts$ are characterized by the condition that their internal vertices are decorated by elements from $\Omega$ only. If two forests have this condition, then their concatenation also has this condition, since an internal vertex of the concatenated forest is an internal vertex of one of the two forests. For the same reason, for any forest $F$ with this condition, any of its subforest $G$ and the quotient forest $F/G$ still have this condition. Thus the subspace $\hrts$ is closed under the concatenation product and the coproduct defined in Eq.~\eqref{eq:cop1}.
\end{proof}

Combining the degree and the grafting operators, we can equip $\rfs$ with other structures.

\begin{defn}\mlabel{de:ofil}
An $\Omega$-operated algebra $(R,P_\Omega)$ with a grading $R=\oplus_{n\geq 0} R^n$ (resp. an increasing filtration $\{R_n\}_{n\geq 0}$) is called an {\bf $\Omega$-operated graded algebra} (resp. {\bf $\Omega$-operated filtered algebra}) if $(R,\oplus_{n\geq 0}R^n)$ is a graded algebra (resp. $(R,\{R_n\}_{n\geq 0})$ is a filtered algebra) and
\begin{equation}
P_\omega (R^n)\subseteq R^{n+1} \quad
\text{(resp. } P_\omega (R_n)\subseteq R_{n+1}).
\mlabel{eq:ofil}
\end{equation}
\end{defn}

As is well-known, a graded algebra $(R,\oplus_{n\geq 0}R^n)$ is a filtered algebra with the filtration $R_n:=\oplus_{k\leq n}R^k, n\geq 0$.
Further by definition, the grafting operator $B_\omega^+$ increases the degree of a rooted forest by one. Thus we have

\begin{lemma}
The $\Omega$-operated algebra $\hrts$ with its grading $\oplus_{n\geq 0}H_{RT}^{n}$ and the associated filtration $H_{RT,(n)}:=\oplus_{k\leq n}H_{RT}^{k}$, is an $\Omega$-operated graded algebra and an $\Omega$-operated filtered algebra.
\mlabel{lem:ofil}
\end{lemma}

For our later applications to \mrbas, we introduce another increasing filtration on $\rfs$ which leads to the notion of the depth of a decorated rooted forest.
For distinction, we use $\calf_{(n)}$ and $H_{RT,(n)}$ for the previous filtration defined by degree and $\calf_n$ and $H_{RT,n}$ for the new filtration defined by depth.

Denote $\bullet_{X}:=\{\bullet_{x}\mid x\in X\}$ and set
\begin{align*}
\calf_0:=M(\bullet_{X})=S(\bullet_{X})\sqcup \{\etree\},
\end{align*}
where $M(\bullet_{X})$ (resp.~$S(\bullet_{X})$) is the submonoid (resp.~subsemigroup) of $\rfs$ generated by $\bullet_{X}$.
Here the using of the notations $M$ and $S$ are justified since $M(\bullet_{X})$ (resp.~$S(\bullet_{X})$) is indeed isomorphic to the free monoid (resp. semigroup) generated by $\bullet_{X}$.
Suppose that $\calf_n$ has been defined for an $n\geq 0$. Then define
\begin{align}
\calf_{n+1}:=M(\bullet_{X}\sqcup (\sqcup_{\omega\in \Omega} B_{\omega}^{+}(\calf_n))).
\mlabel{eq:frec}
\end{align}
Thus we obtain $\calf_n\subseteq \calf_{n+1}$ and
\begin{align}
\rfs = \lim_{\longrightarrow} \calf_n=\bigcup_{n=0}^{\infty}\calf_n.
\mlabel{eq:rdeff}
\end{align}
Elements $F\in \calf_n\setminus \calf_{n-1}$ are said to have {\bf depth} $n$, denoted by $\dep(F) = n$.
Here are some examples:
\begin{align*}
\dep(\etree) =&\ \dep(\bullet_x) =0,\ \dep(\bullet_\omega)=\dep(B^+_{\omega}(\etree)) = 1,\  \dep(\tddeux{$\omega$}{$\alpha$})= \dep(B^+_{\omega}(B^+_{\alpha}(\etree))) =2, \\
\dep(\tdun{$x$}\tddeux{$\omega$}{$y$}\tdun{$y$}) =&\ \dep(\tddeux{$\omega$}{$y$})=
\dep(B^+_{\omega}(\bullet_y)) =1, \ \dep(~\tdtroisun{$\omega$}{$x$}{$\alpha$}) = \dep(B^+_{\omega}(B^+_{\alpha}(\etree) \bullet_x)) = 2,
\end{align*}
where $\alpha,\omega\in \Omega$ and $x, y \in X$.

Note the subtle difference of this depth from the usual notion of depth of a rooted tree, defined to be the length of the longest path from a root of $F$ to its leafs. The advantage of this new depth is that it is consistent with the natural depth of bracketed words to be introduced in Section~\mref{ss:freeoa}. This difference is most easily seen in $\dep(\bullet_x)=0$ while $\dep(\bullet_\omega)=1$. In general our notion $\dep(F)$ of depth for $F\in \calf$ is the same as the usual depth when all the longest paths from a root of $F$ to the leafs end in vertices with decorations from $X$; while $\dep(F)$ is the usual depth of $F$ plus one if one of these longest paths ends in a vertex with decoration from $\Omega$.

\subsection{Free $\Omega$-cocycle Hopf algebras of decorated planar rooted forests}
In this subsection, we will combine the notions of $\Omega$-operated algebras and Hopf algebras to define an $\Omega$-(operated) Hopf algebra and $\Omega$-(operated) cocycle Hopf algebra.
We then show that $\hrts$ is a free $\Omega$-cocycle Hopf algebra on the set $X$.
For this purpose, we recall the following concepts.

\begin{defn}\cite{Guo09} Let $\Omega$ be a nonempty set.
\begin{enumerate}
\item
 An {\bf $\Omega$-operated algebra} is an algebra $A$ together with a family of linear operators $P_{\omega}: A\to A$, $\omega\in \Omega$.
\item
Let $(A,\, \Po)$ and $(A',\,\Pop)$ be  $\Omega$-operated algebras.
A linear map $\phi : A\rightarrow A'$ is called an {\bf $\Omega$-operated algebra homomorphism} if $\phi$ is an algebra homomorphism such that $\phi P_\omega = P'_\omega \phi$ for $\omega\in \Omega$.
\item
A {\bf free $\Omega$-operated algebra on a set $X$} is an $\Omega$-operated algebra $(A, \Po)$ together with a set map $j_{X}: X\rightarrow A$ with the property that, for any $\Omega$-operated algebra $(A', \Pop)$ and any set map $f: X\rightarrow A'$, there is a unique homomorphism $\bar f:A\to A'$ of $\Omega$-operated algebras such that $\bar{f} j_X=f$.
\end{enumerate}
\end{defn}

Now we enrich these notions by adding the bialgebra structures.
\begin{defn}
\begin{enumerate}
\item An {\bf $\Omega$-operated bialgebra} is a bialgebra $(H,m,1_H,\Delta, \varepsilon)$ which is also an $\Omega$-operated algebra $(H,\,\Po)$.
\item Let $(H,\,\Po)$ and $(H',\,\Pop)$ be $\Omega$-operated bialgebras.
A linear map $\phi : H\rightarrow H'$ is called an {\bf $\Omega$-operated bialgebra homomorphism} if $\phi$ is a bialgebra homomorphism such that $\phi  P_\omega = P'_\omega \phi$ for $\omega\in \Omega$.
\item
An {\bf$\Omega$-cocycle bialgebra} is an $\Omega$-operated bialgebra $(H,m,1_H,\Delta, \varepsilon,\,\Po)$ which satisfies the cocycle condition:
\begin{equation}
\Delta P_\omega=P_\omega\ot 1_H + (\id\ot P_\omega)\Delta \, \text{ for }\omega \in \Omega.
\mlabel{eq:cocycle2}
\end{equation}
If the bialgebra in an $\Omega$-cocycle bialgebra is a Hopf algebra, then it is called an {\bf$\Omega$-cocycle Hopf algebra}.
\item
The {\bf free $\Omega$-cocycle bialgebra on a set $X$} is an $\Omega$-cocycle bialgebra $(H_X,\, m_X,\, 1_X,\, \Delta_X,\,  \varepsilon_X,\,\\ \Po)$ together with a set map $j_X:X\to H_X$ with the property that, for any cocycle bialgebra $(H,\, m,\, 1_H,\,  \Delta, \, \varepsilon,\,\Pop)$ and set map $f:X\to H$ whose images are primitive (that is, $\Delta(f(x))=f(x)\ot 1_H+1_H\ot f(x)$), there is a unique homomorphism $\free{f}:H_X\to H$ of $\Omega$-operated bialgebras such that $\free{f} j_X=f$. The concept of a {\bf free $\Omega$-cocycle Hopf algebra} is defined in the same way.
\mlabel{it:def4}
\end{enumerate}
\mlabel{de:decHopf}
\end{defn}

We are indebt to Foissy for the following result.

\begin{lemma}
Let $(H,\, m,\,1_H,\, \Delta,\,\varepsilon,\, \Po)$
be an $\Omega$-cocycle bialgebra. Then for each $\omega\in \Omega$,
$$P_{\omega} (H) := \{P_{\omega} (h) \mid h\in H\}$$ is a coideal of $H$.
\mlabel{lem:coidealv}
\end{lemma}

\begin{proof}
Let $\omega\in \Omega$. We first show $P_{\omega} (h)\subseteq \ker \varepsilon$.
Let $h':=P_{\omega}(h) \in H$ be arbitrary with $h \in H$. Using the Sweedler notation, we can write
\begin{align*}
\varepsilon(h')&=\varepsilon\biggl(\sum_{(h')} h'_{(1)} \varepsilon( h'_{(2)})\biggr)
=\sum_{(h')} \varepsilon(h'_1)  \varepsilon( h'_2)
=(\varepsilon\ot \varepsilon) \Delta(h').
\end{align*}
Thus
\begin{align*}
\varepsilon(h')&=\varepsilon P_{\omega}(h)=(\varepsilon\ot \varepsilon) \Delta(h')=(\varepsilon\ot \varepsilon) \Delta(P_{\omega}(h))\\
&=(\varepsilon\ot \varepsilon)\Big(P_{\omega}(h)\ot 1_H+(\id \ot P_{\omega})\Delta(h)\Big)\quad(\text{by Eq.~(\mref{eq:cocycle2})})\\
&=(\varepsilon\ot \varepsilon)(P_{\omega}(h)\ot 1_H)+(\varepsilon\ot \varepsilon P_{\omega})\Delta(h)\\
&=\varepsilon P_{\omega}(h)+\sum_{(h)}\varepsilon(h_{(1)})\varepsilon P_{\omega}(h_{(2)})\\
&=\varepsilon P_{\omega}(h)+\varepsilon P_{\omega}\Big(\sum_{(h)}\varepsilon(h_{(1)})(h_{(2)})\Big)\\
&=\varepsilon P_{\omega}(h)+\varepsilon P_{\omega}(h),
\end{align*}
which implies that $\varepsilon P_{\omega}(h)=0$.

Secondly, for any $h\in H$,
\begin{align*}
\Delta(P_{\omega}(h))&=P_{\omega}(h)\ot 1_H+(\id \ot P_{\omega})\Delta(h) \quad \text{ (by Eq.~(\mref{eq:cocycle2}))}\\
&= P_{\omega}(h)\otimes 1_{H}+(\id \ot P_{\omega}) \bigg( \sum_{(h)}(h_{(1)}\ot h_{(2)}) \bigg)\\
&= P_{\omega}(h)\otimes 1_{H}+ \sum_{(h)}h_{(1)}\ot P_{\omega}(h_{(2)}) \in P_{\omega}(H) \otimes H+ H\otimes P_{\omega}(H).
\end{align*}
Thus $P_{\omega}(H)$ is a coideal.
\end{proof}

The following result generalizes the universal properties of several related structures studied in~\cite{CK98, Guo09, Moe01, ZGG16}. See~\cite[Theorem~2.3]{CGPZ18} for the commutative case.

\begin{lemma}\mcite{ZCGL18}
Let $j_{X}: X\hookrightarrow \hrts$, $x \mapsto \bullet_{x}$ be the nature embedding and $m_{RT}$ be the concatenation product.
The quadruple $(\hrts, \,\mul,\,1, \, \Bo)$ together with $j_X$ is the free $\Omega$-operated algebra on $X$.
\mlabel{lem:propm}
\end{lemma}

We next strengthen Lemma~\mref{lem:propm} to include the bialgebra structure.
\begin{theorem}
Let $j_{X}: X\hookrightarrow \hrts$, $x \mapsto \bullet_{x}$ be the nature embedding and $m_{RT}$ be the concatenation product.
\begin{enumerate}
\item
 The sextuple $(\hrts, \,\mul,\,1, \, \Delta_{RT},\,\epl, \,\Bo)$ together with $j_X$ is the free $\Omega$-cocycle  bialgebra on $X$.  \mlabel{it:fubialg}

\item The Hopf algebra given by the connected graded bialgebra $(\hrts, \,\mul,\,1, \, \Delta_{RT},\,\epl, \\
\, \Bo)$ together with $j_X$ is the free $\Omega$-cocycle Hopf algebra on $X$.
 \mlabel{it:fuhopf}
\end{enumerate}
\mlabel{thm:propm}
\end{theorem}

\begin{proof}
(\mref{it:fubialg})
By Theorem~\mref{thm:hopfx}, the quintuple $(\hrts, \,\mul,\,1,\, \Delta_{RT}, \epl)$ is a bialgebra.
Furthermore, by Eq.~(\mref{eq:cocycle}), the sextuple $(\hrts, \,\mul,\,1,\, \Delta_{RT},\epl, \,\Bo)$ is an $\Omega$-cocycle  bialgebra.

To verify the freeness, let $(H,\, m,\,1_H,\, \Delta,\,\varepsilon,\, \Po)$ be an $\Omega$-cocycle bialgebra and $f: X\rightarrow H$ a set map such that
\begin{align}
\Delta(f(x))=f(x)\ot 1_H+1_H\ot f(x) \, \text{ for all }\, x\in X.
\mlabel{eq:prim}
\end{align}
In particular, $(H,\, m,\, 1_H,\, \Po)$ is an $\Omega$-operated algebra.
It follows from Lemma~\mref{lem:propm} that there exists a unique $\Omega$-operated algebra homomorphism $\free{f}:\hrts \to H$ such that $\free{f} j_X={f}$. It remains to check the following two compatibilities between the coproducts $\Delta$ and $\col$, and between the counit $\varepsilon$ and $\epl$.
\begin{align}
\Delta \free{f} (F)&=(\free{f}\ot \free{f}) \col (F), \mlabel{eq:copcomp}\\
\varepsilon \free{f} (F)&=\epl (F) \quad  \text{for all } F\in \rfs. \mlabel{eq:counit}
\end{align}

To verify Eq.~(\mref{eq:copcomp}), we consider the set
\begin{align*}
\mathscr{A} := \{ F\in \hrts \, | \, \Delta(\free{f}(F))=(\free{f} \ot \free{f}) \col(F) \}.
\end{align*}
By Lemma~\mref{lem:propm}, $\hrts$ is generated by $X$ as an $\Omega$-operated algebra. Thus to verify Eq.~(\mref{eq:copcomp}) we just need to show that $\mathscr{A}$ is an $\Omega$-operated subalgebra of $\hrts$ that contains $X$.

Since $\free{f}$ is an $\Omega$-operated algebra homomorphism, and $\col$ and $\Delta$ are algebra homomorphisms from $\hrts$ and $H$, respectively, we get $1 \in \mathscr{A}$ and $\mathscr{A}$ is a subalgebra of $\hrts$. For any $x \in X$, we have
\begin{align*}
\Delta(\free{f}(\bullet_{x}))=&\ \Delta(f(x))\\
=&\ f(x) \ot 1_H +1_H \ot f(x)  \,\, \text{(by Eq.~(\mref{eq:prim}))}\\
=&\ \free{f} (\bullet_{x})\ot \free{f}(1)+\free{f}(1)\ot \free{f}(\bullet_{x}) \\
=&\ (\free{f} \ot \free{f})(\bullet_{x} \ot 1 +1 \ot \bullet_{x} )\\
=&\ (\free{f} \ot \free{f})\col(\bullet_{x}).
\end{align*}
Thus $\bullet_{x} \in \mathscr{A}$. Further for any $F \in \mathscr{A}$ and $\omega \in \Omega$, we have
\begin{align*}
&\ \Delta \free{f} (B_{\omega}^{+}({F}))\\
=&\ \Delta P_\omega(\free{f} ({F}))
\quad(\text{by $\bar{f}$ being an $\Omega$-operated algebra homomorphism}) \\
=&\ P_\omega(\free{f}({F}))\ot 1_{H}+ (\id\ot P_\omega)\Delta(\free{f} ({F}))\quad(\text{by Eq.~(\mref{eq:cocycle2})})\\
=&\ P_\omega(\free{f}({F}))\ot 1_{H}+ (\id\ot P_\omega)(\free{f}\ot \free{f}) \col ({F}) \quad \text{(by $F \in \mathscr{A}$)}\\
=&\ P_\omega(\free{f}({F}))\ot 1_{H}+ (\free{f}\ot P_\omega\free{f}) \col ({F})\\
=&\ \free{f}(B_{\omega}^{+}({F}))\ot 1_{H}+ (\free{f}\ot \free{f}B_{\omega}^+) \col ({F}) \ \ (\text{by $\bar{f}$ being an $\Omega$-operated algebra homomorphism}) \\
=&\ (\free{f}\ot \free{f})\Big(B_{\omega}^{+}({F})\ot 1+(\id\ot B_{\omega}^+)\col ({F})\Big) \\
=&\ (\free{f}\ot \free{f}) \col (B_{\omega}^+({F}))\\
=&\ (\free{f}\ot \free{f}) \col (F).
\end{align*}
Thus $\mathscr{A}$ is stable under  $B_{\omega}^{+}$ for any $\omega \in \Omega$ and so $\mathscr{A}=\hrts$.

Similarly, to verify Eq.~(\mref{eq:counit}), we just need to show that  the subset
\begin{align*}
\mathscr{B}:=\{ F \in \hrts \mid \, \varepsilon(\free{f}(F))= \varepsilon_{\mathrm{RT}}(F)\} \subseteq \hrts.
\end{align*}
is an $\Omega$-operated subalgebra of $\hrts$.

Since $\free{f}$ is an $\Omega$-operated algebra homomorphism, $\varepsilon_{\mathrm{RF}}$ and $\varepsilon$ are algebra homomorphisms from $\hrts$ and $H$, respectively. So we get $1 \in \mathscr{B}$ and $\mathscr{B}$ is a subalgebra of $\hrts$. For any $x \in X$, by Eq.~(\mref{eq:prim}) and the left counicity, we obtain
\begin{align*}
(\varepsilon \ot \id) \Delta(f(x))= \varepsilon(f(x)) 1_H +1_\bfk \ot f(x) =\beta_{\ell}(f(x)),
\end{align*}
which implies that $\varepsilon(f(x))=0$. Then
\begin{align*}
\varepsilon(\free{f}(\bullet_{x}))= \varepsilon(f(x))=0=\varepsilon_{\mathrm{RT}}(\bullet_{x}),
\end{align*}
showing $\bullet_{x} \in \mathscr{B}$. For $F \in \mathscr{B}$ and $\omega \in \Omega$, we have
\begin{align*}
\varepsilon(\free{f}(B_{\omega}^{+}(F)))&\ = \varepsilon(P_{\omega}(\free{f}(F))) \,\, \text{(by $\free{f}$ being an $\Omega$-operated algebra homomorphism)}\\
&\ =\varepsilon P_\omega(\bar{f}({F})) =0\quad(\text{by Lemma~\mref{lem:coidealv}})\\
&\ =\varepsilon_{\mathrm{RT}}(B_{\omega}^{+}(F)).
\end{align*}
Hence $\mathscr{B}$ is stable under $B_{\omega}^{+}$ for any $\omega \in \Omega$ and so $\mathscr{B}= \hrts$. This completes the proof.

(\mref{it:fuhopf})
The proof follows from Item~(\mref{it:fubialg}) and the well-known fact that
any bialgebra homomorphism between two Hopf algebras is compatible with the antipodes~\cite[Lemma~4.04]{Swe69}.
\end{proof}

If $X = \emptyset$, we obtain the freeness of $\hck(\emptyset, \Omega)=H_{RT}(\Omega)$, which is the decorated noncommutative Connes-Kreimer Hopf algebra by Remark~\mref{re:3ex}~(\mref{it:2ex}).

\begin{coro}
The sextuple $(H_{RT}(\Omega), \,\mul,\,1, \, \Delta_{RT},\,\epl, \,\Bo)$ is the free $\Omega$-cocycle Hopf algebra on the empty set, that is, the initial object in the category of $\Omega$-cocycle Hopf algebras.
\mlabel{cor:rt16}
\end{coro}

Further taking $\Omega=\{\omega\}$ to be a singleton in Corollary~\mref{cor:rt16}, all decorated planar rooted forests have the same decoration and hence can be rendered undecorated as in the Foissy-Holtkamp Hopf algebra~\mcite{Foi02, Hol03}. Thus we have, similar to~\mcite{CK98,Moe01,ZGG16},

\begin{coro}
Let $\mathcal{F}$ be the set of planar rooted forests without decorations.
Then sextuple $(\bfk\mathcal{ F}, \,\mul,\,1, \, \Delta_{RT},\,\epl, \,B_{\omega}^+)$ is the free cocycle Hopf algebra on the empty set,
that is, the initial object in the category of cocycle Hopf algebras.
\mlabel{coro:rt16}
\end{coro}

\subsection{Free $\Omega$-operated monoids and algebras}
\mlabel{ss:freeoa}
Our next goal is to construct free matching Rota-Baxter algebras by applying the method of Gr\"obner-Shirshov bases which works better in the context of bracketed words. Thus
in this subsection, we recall the construction of a free $\Omega$-operated monoid and $\Omega$-operated algebra in terms of bracketed words on a set $X$ and identify them with the free $\Omega$-operated algebra $H_{RT}(X,\Omega)=\bfk\rfs$. See~\mcite{Guo09} for more details of these bracketed words.

Given an $\omega \in \Omega$ and a set $Y$, let $\lfloor Y \rfloor_{\omega}$
denote the set $\big\{\lfloor y\rfloor_{\omega}\mid y\in Y\big\}$, so it is indexed by $Y$ but disjoint with $Y$. We also assume that the sets $\lfloor Y \rfloor_{\omega}$ to be disjoint with each other as $\omega$ varies in $\Omega$.

We now define the free $\Omega$-operated monoid over the set $X$ as the limit of a direct system
\begin{align*}
\{i_{n,\,n+1}: \frak{M}_{n}\rightarrow \frak{M}_{n+1}\}_{n=0}^{\infty}
\end{align*}
of inductively defined free monoids $\mathfrak{M}_{n}$, where the transition homomorphisms $i_{n+1,\,n}$ are natural embeddings.
For the initial step of $n=0$, we define $\frak{M}_{0}:=M(X)$ and then define
\begin{equation*}
\mathfrak{M}_{1}:= M \big(X\sqcup (\sqcup_{\omega\, \in \Omega}\lfloor \frak{M}_{0} \rfloor_{\omega})\big)
\end{equation*}
with the natural embedding
\begin{equation*}
i_{0,\,1}: \frak{M}_{0}=M(X)\hookrightarrow  \frak{M}_{1}=M\left(X\sqcup (\sqcup_{\omega\,\in \Omega}\lfloor \frak{M}_{0}\rfloor_{\omega})\right).
\end{equation*}
Note that $\lfloor \frak{M}_{0}\rfloor_{\omega}\subseteq \frak{M}_{1}$ for each $\omega\in \Omega$. In particular, $1\in \frak{M}_{0}$ is sent to $1\in \frak{M}_{1}$.
Inductively assume that, for any given $n \geq 1$, $\frak{M}_{k}, k\geq n$ with the natural embedding
\begin{equation}
i_{n-1,\,n}: \frak{M}_{n-1}\hookrightarrow \frak{M}_n
\mlabel{eq:inclu2}
\end{equation}
have been defined. We then define
\begin{equation}
\frak{M}_{n+1}:=M\left(X\sqcup(\sqcup_{\omega\,\in \Omega}\lfloor \frak{M}_n\rfloor_{\omega})\right).
\mlabel{eq:wrec}
\end{equation}
The natural embedding in Eq.~(\mref{eq:inclu2}) induces the natural embedding
$$\lfloor \frak{M}_{n-1}\rfloor_{\omega}\hookrightarrow \lfloor \frak{M}_n\rfloor_{\omega},$$
yielding a monomorphism of free monoids
$$i_{n,\,n+1}: \frak{M}_n=M\left(X\sqcup(\sqcup_{\omega \,\in \Omega}\,\lfloor \frak{M}_{n-1}\rfloor_{\omega})\right)
\hookrightarrow M\left(X\sqcup(\sqcup_{\omega \,\in \Omega}\,\lfloor \frak{M}_n\rfloor_{\omega})\right)=\frak{M}_{n+1}.$$
This completes the inductive construction of the direct system. Finally we define the direct limit of monoids
$$ \frak{M}(X,\,\Omega):=\lim_{\longrightarrow}\frak{M}_{n}=\bigcup_{n\geq 0}\frak{M}_{n}$$
with identity $1$.
Elements in $\frak{M}(X,\,\Omega)$ are called {\bf $\Omega$-bracketed words} in $X$ and elements of $\frak{M}_{n}\backslash \frak{M}_{n-1}$ are said to have {\bf depth} $n$, denoted by $\dep_\frakM(w)=n$. Define
\begin{align*}
P_{\omega}:\frak{M}(X,\,\Omega)\rightarrow \frak{M}(X,\,\Omega),\, u\mapsto \lfloor u \rfloor_{\omega}, \, \omega\in \Omega,
\end{align*}
and extend it by linearity to a linear operator on $\bfk\frak{M}(X,\,\Omega)$, still denoted by $P_{\omega}$.
Then the pair $\big(\frak{M}(X,\,\Omega), \Po \big)$ is an $\Omega$-operated monoid and its linear span $\big(\bfk\frak{M}(X,\,\Omega), \Po \big)$ is an $\Omega$-operated algebra.

Let $X$ be a set and $\Omega$ a nonempty set disjoint with $X$. Taking direct limit in Eq.~(\mref{eq:wrec}) we obtain
\begin{equation}
\frakM(X,\Omega)=M\left(X \sqcup \left(\sqcup_{\omega\in \Omega} \lc \frakM(X,\Omega)\rc\right)\right).
\mlabel{eq:wset}
\end{equation}
Thus any $1\neq u \in \frak{M}(\Omega,\, X)$ has a unique factorization \begin{equation}
u= w_{1} \cdots w_{k}, \quad w_{i}\in X\cup \frak{M}(\Omega,\,  X), 1\leq i\leq k,\, k\geq 1.
\mlabel{eq:wfact}
\end{equation}
We call $k$ the {\bf breadth} of $u$ and denote it by $|u|$. For $u=1 \in \frak{M}(\Omega,\,  X),$ we define $|u|:=0.$

\begin{prop} \cite[Corollary~3.6]{Guo09}
Let $X$ be a set and $\Omega$ a nonempty set. Let $j_{X}: X\rightarrow \bfk\frak{M}(X,\,\Omega)$ be the natural embedding and let $\cdot$ be the concatenation product. Then the triple $(\bfk\frak{M}(X,\,\Omega),\, \cdot,\, \Po)$ together with $j_X$ is
the free $\Omega$-operated algebra on $X$.
\mlabel{pp:freew}
\end{prop}

Lemma~\mref{lem:propm} and the uniqueness of the free objects in the category of $\Omega$-operated algebras then yield the isomorphism of $\Omega$-operated algebras
\begin{align}
\theta: ({\bfk}\mathfrak{M}(X,\,\Omega), \,\cdot,\, \Po)\cong ({\bfk}\rfs, \,\cdot,\,\Bo),
\mlabel{eq:Isomorphism2}
\end{align}
sending $x\in X$ to $\theta(x):= \bullet_x$.
Comparing Eqs.~(\mref{eq:frec}) and (\mref{eq:wrec}), we see that $\theta$ preserves the filtrations of bracketed words in $\frakM(X,\Omega)$ and forests in $\calf(X,\Omega)$ given by depths:
\begin{equation}
\theta(\frakM_n) =\calf_n, \quad n\geq 0.
\mlabel{eq:depcomp}
\end{equation}

Further, for $w\in \frakM(X,\Omega)$, let $\deg_{td}(w)$, called the {\bf total degree of $w$}, denote the total number (counting multiplicities) of the appearances of elements of $X$ and brackets $\lc \cdot \rc_\omega, \omega\in \Omega,$  in $w$. So $\lc x y \lc x\rc_\alpha z\rc_\alpha$ has $\deg_{td}(w)=6$ since the letters appear four time and the operators appear twice.

For $n\geq 0$, let $\frakM^{(n)}$ denote the subset of $\frakM(X,\Omega)$ with total degree $n$ and let $\frakM_{(n)}$ denote the union $\cup_{k\leq n}\frakM^{(k)}$. Then we have a grading and a filtration
\begin{equation}
\bfk \frakM(X,\Omega)=\oplus_{n\geq 0} \bfk \frakM^{(n)}, \quad
\bfk \frakM_{(n)} \subseteq \bfk \frakM_{(n+1)}, \quad n\geq 0.
\mlabel{eq:wdegfil}
\end{equation}

Since the map $\theta$ sends $x\in X$ to $\bullet_x$ and $\lc w\rc_\omega$ to $B_\omega^+(\theta(w))$, it preserves the degrees: $\deg_{tw}(w)=\deg(\theta(w))$, and the resulting gradings and filtrations.
Thus as a consequence of Lemma~\mref{lem:propm}, we have
\begin{coro}
With the grading and its associated filtration on $\bfk\frakM(X,\Omega)$ defined by the total degree $\deg_{td}$ in Eq.~\eqref{eq:wdegfil}, the free $\Omega$-operated algebra $\bfk \frakM(X,\Omega)$ is an $\Omega$-operated graded algebra and an $\Omega$-operated filtered algebra, isomorphic to the ones for $\bfk\calf(X,\Omega)$ in Lemma~\mref{lem:propm}.
\mlabel{co:wfil}
\end{coro}

\section{Gr\"obner-Shirshov bases and free \mrbas}
\mlabel{sec:freemrb}
In this section we construct free \mrbas from bracketed words and decorated rooted forests by the method of Gr\"{o}bner-Shirshov bases.
We begin with a brief review of \mrbas emphasizing their many connections. We then recall the Composition-Diamond (CD) Lemma for the Gr\"obner-Shirshov bases of operated algebras. With these preparations, the Gr\"obner-Shirshov bases for \mrbas is then obtained. This gives the desired construction of free \mrbas in terms of bracketed words. We finally apply the isomorphism in Eq.~\eqref{eq:Isomorphism2} to give a construction of free \mrbas in terms of decorated rooted forests.

\subsection{\Mrbas }
\mlabel{ss:mrba}
In this subsection, we recall the concept of \mrbas, which generalizes that of Rota-Baxter algebras.

\begin{defn}\mcite{GGZ19}
Let $\Omega$ be a nonempty set and let
$\lambda_\Omega:=(\lambda_\omega)_{\omega\in \Omega}\subseteq \bfk$
be a parameterized family of scalars with index set $\Omega$. More precisely, $\lambda_\Omega$ is a map $\Omega\to \bfk$.
\begin{enumerate}
\item A {\bf matching (multiple) \rba} of weight $\lambda_\Omega$ is a pair $(R, P_\Omega)$ consisting of an algebra $R$ and a family
$P_\Omega:=(P_\omega)_{\omega\in \Omega}$
of linear operators
$P_\omega: R\longrightarrow R, \omega\in \Omega\,,$
that satisfy the {\bf matching Rota-Baxter equation}
\begin{align}
P_\alpha(x)P_{\beta}(y)&=P_{\alpha}(xP_{\beta}(y)) +P_{\beta}(P_{\alpha}(x)y)+\lambda_\beta P_{\alpha}(xy)
\, \tforall x,y \in R, \alpha,\beta\in \Omega\,.
\label{eq:RBid}
\end{align}
When $\lambda_\Omega=\{\lambda\}$ in $(R,\lambda_\Omega)$, that is, when $\lambda_\Omega: \Omega\to \bfk$ is constant, we also call the \match \rba to have weight $\lambda$.

\item Let $(R,\, P_\Omega)$ and $(R',\,{P'}_{\Omega})$ be \mrbas of the same weight $\lambda_\Omega$.
A linear map $\phi : R\rightarrow R'$ is called a {\bf \mrba homomorphism} if $\phi$ is an algebra homomorphism  such that $\phi  P_\omega = P'_\omega \circ\phi$ for all $\omega\in \Omega$.
\end{enumerate}
\end{defn}

To motivate of our study of free matching Rota-Baxter algebras, we list some properties of \mrbas and refer the reader to~\mcite{GGZ19,GGZy,GGL} for further details.
\begin{enumerate}
\item
Any Rota-Baxter algebra of weight $\lambda$ can be viewed as a \mrba
of weight $\lambda$ by taking $\Omega$ to be a singleton.
\item
When $\lambda=0$,
the \match Rota-Baxter equation is Lie compatible in the sense that
\begin{align*}
[P_{\alpha}(x), P_{\beta}(y)]
=P_{\alpha}([x, P_{\beta}(y)])+P_{\beta}([P_{\alpha}(x), y]).
\end{align*}
Here the Lie bracket is taking as the commutator. In the case when $|\Omega|=2$, this has been studied in~\mcite{St}.
\item \cite[Proposition~2.5]{GGZ19} \Match Rota-Baxter algebras provide a solution to the linearity of the set of Rota-Baxter operators on an algebra as follows. Let $(R,\, \Po)$ be a \mrba of weight $\lambda$. Then any finite linear combination
\begin{align*}
P:=\sum_{\omega\in \Omega} k_\omega P_\omega,\,  k_\omega \in \bfk,
\end{align*}
with $k_\omega\in \bfk$ is a Rota-Baxter operator of weight $\lambda \sum_\omega k_\omega$. In particular, if $\sum_\omega k_\omega=1$, then $P$ is a Rota-Baxter algebra of weight $\lambda$.
Thus any element in the linear span $\sum_{\omega\in \Omega} \bfk P_\omega$ of $P_\Omega$ is a Rota-Baxter operator of certain weight.
\item \cite[Corollary~4.5]{GGZ19} \Match Rota-Baxter algebras have a close connection with \match (multiple) pre-Lie algebras introduced by Foissy~\mcite{Foi18}.
Let $(R, \Po)$ be a \mrba of weight $\lambda_\Omega$.  Define
\begin{align*}
x \ast_\omega y:= P_{\omega}(x)y-yP_{\omega}(x)-\lambda_\omega yx\, \text{ for } x,y,z\in R, \omega\in \Omega.
\end{align*}
Then the pair $(R, (\ast_\omega)_{\omega\in \Omega})$ is a \match (multiple) pre-Lie algebra.
\item \cite[Theorem~3.4]{GGZ19} A \match \rba $(R, \,\Po)$ of weight $\lambda_\Omega$ induces a \match dendriform algebra $(R, \, (\prec_{\omega})_{\omega \in \Omega}, (\succ_{\omega})_{\omega \in \Omega})$, where
\begin{equation*}
x\prec_{{\omega}}y := xP_{\omega}(y)+\lambda_\omega xy,\,\, x\succ_{{\omega}}y := P_{\omega}(x)y \, \text{ for }\, x,y \in R, \omega\in \Omega.
\end{equation*}
\item
\cite[Example~2.3]{GGZ19}, \cite{GGL}
Consider the $\RR$-algebra $R:=\mathrm{Cont}(\RR)$ of continuous functions on $\RR$. Let $K_\omega(x,t)$ be a parameterized family of kernels of continuous functions on $\RR^2$ and let
\begin{equation} I_\omega: R\longrightarrow R, \quad
f(x)\mapsto \int_0^x K_\omega(x,t)f(t)\,dt, \quad \omega \in \Omega,
\mlabel{eq:mint}
\end{equation}
be the corresponding family of Volterra integral operators~\mcite{Ze}.
Then when $K_\omega(x,t)$ is independent of $x$, the pair $(R,(I_\omega)_{\omega\in \Omega})$ is a \mrba of weight zero.
\mlabel{ex:int}
\item
\cite[\S~2.2]{GGZ19}
For $r, s\in R\ot R$, let
$$ r_{13} s_{12} - r_{12}s_{23}+r_{23}s_{13}=-\lambda s_{13}$$
be the {\bf polarized associative Yang-Baxter equation} of weight $\lambda$. Then a solution of this equation gives a matching Rota-Baxter operator of weight $\lambda$.
\end{enumerate}

The purpose of this section is to construct free \mrbas. Since by Proposition~\mref{pp:freew}, $\bfk\frakM(X,\Omega)$ is the free $\Omega$-operated algebra on a set $X$, the free \mrba on $X$ is obtained by taking the quotient of $\bfk\frakM(X,\Omega)$ modulo the operated ideal generated by the \mrba relations. More precisely, let $\Id(S)$ be the operated ideal of $\bfk\frakM(X,\Omega)$ generated by the set
\begin{equation}
S:=\left\{\left .\lf x\rf_\alpha\lf y\rf_\beta-\lf x\lf y\rf_\beta\rf_{\alpha}-\lf\lf x\rf_\alpha y\rf_{\beta}-\lambda_\beta\lf xy\rf_{\alpha}\,\right |\,x,y\in \frakM(\Omega,\,X), \alpha, \beta \in \Omega\right\}.
\mlabel{eq:gsw}
\end{equation}
Then the free \mrba $\ncshaw(X,\Omega)$ is given by the quotient $\bfk\frakM(X,\Omega)/\Id(S)$. We will identify a canonical subset of $\frakM(X,\Omega)$ which gives a linear basis of this quotient and express the operation of the \mrba in terms of this basis.

Of course the free \mrba is also given by taking the quotient of the other realization $\bfk \calf(X,\Omega)$ of the free $\Omega$-operated algebra on the set $X$ modulo the \mrba relations: $\bfk\calf(X,\Omega)/\Id(\frakS)$, where $\Id(\frakS)$ is the $\Omega$-operated ideal of $\bfk\calf(X,\Omega)$ generated by
\small{
\begin{equation}
\frakS:=
\left\{B_\alpha^+(F_1)\dl B_\beta^+(F_2)-B_{\alpha}^+\Big(F_1 \dl B_\beta^+(F_2) \Big)-B_{\beta}^+\Big(B_\alpha^+(F_1)\dl F_2\Big)-\lambda_\beta B_{\alpha}^+(F_1\dl F_2)|\alpha, \beta \in \Omega\right\}.
\mlabel{eq:gsf}
\end{equation}
}
We choose to work with $\bfk\frakM(X,\Omega)$ since it provides a simpler context to apply the method of \gsbs, as we will carry out next.

\subsection{Gr\"{o}bner-Shirshov bases of free \mrbas}
In this subsection, we recall the Composition-Diamond Lemma for the $\Omega$-operated (unitary) algebra $\bfk\frak{M}(X,\,\Omega)$ and apply it to construct a linear basis of the free \mrba on a set.

\subsubsection{Composition-Diamond Lemma for free $\Omega$-operated algebras}
For further details on the notations and background, we refer the reader to~\mcite{BCQ10, GG, GSZ}.

Let $X$ be a set and $\Omega$ a nonempty set, $\star \notin X$, and $X^{\star}:=X\sqcup \{\star\}.$
By a {\bf $\star$-bracketed word} on $X$, we mean any bracketed word in $\frakM(\Omega,\, X)^{\star}:=\mathfrak{M}(\Omega,\,X^{\star})$ with exactly one occurrence of $\star$, counting multiplicities.
For $q\in \mathfrak{M}(\Omega,\,X)^{\star}$ and $u\in \mathfrak{M}(\Omega,\,X)$, we define
$$q|_{u}:=q|_{\star\mapsto u}$$
to be the bracketed word on $X$ obtained by replacing the unique occurrence of $\star$ in $q$ by $u$.
For $q\in \mathfrak{M}(\Omega,\,X)^{\star}$ and $s=\Sigma_ic_iq|_{u_i}\in \bfk\mathfrak{M}(\Omega,\,X),$ where $c_i\in \bfk$ and $u_i\in \mathfrak{M}(\Omega,\,X)$, we define
    $$q|_s:=\Sigma_ic_iq|_{u_i},$$
  and extend this notation to any $q\in \bfk\mathfrak{M}^{\star}(\Omega,\,X)$ by linearity. Note that the element $q|_{s}$ is usually not a bracketed word but a bracketed polynomial.

A {\bf monomial order} on $\mathfrak{M}(\Omega,\,X)$ is a well order $\leq$ on $\mathfrak{M}(\Omega,\,X)$ such that
\begin{equation}
u < v \Rightarrow q|_{u} < q|_{v},\, \text{for all}\,\, u,v \in \mathfrak{M}(\Omega,\,X)\,\text{and all}\, q\in
 \mathfrak{M}^{\star}(\Omega,\,X).
 \mlabel{eq:order}
\end{equation}
Here, as usual, we denote $u < v$ if $u\leq v$ but $u\neq v.$ Since $\leq$ is a well order, it follows from Eq.~(\mref{eq:order}) that $1\leq u$ and $u<\lfloor u \rfloor_{\omega}$ for all $u \in \mathfrak{M}(\Omega,\,X)$ and $\omega \in \Omega$.

Let $\leq$ be a monomial order on $\mathfrak{M}(\Omega,\,X)$ and let $f\in \bfk\mathfrak{M}(\Omega,\,X)$.
\begin{enumerate}
\item If $f\notin \bfk$, the unique largest monomial $\bar{f}$ appearing in $f$ is called the {\bf leading bracketed word (monomial)} of $f$.
\item The coefficient of $\bar{f}$ in $f$ is called the {\bf leading coefficient} of $f$, which is denoted by $c(f)$.
\item  If $f\notin \bfk$ and $c(f)=1$, then $f$ is  {\bf monic with respect to the monomial order} $\leq$ and
a subset $S \subset  \bfk\mathfrak{M}(\Omega,\,X)$ is {\bf monic with respect to} $\leq$ if every $s \in S$ is monic with respect to $\leq.$
\end{enumerate}

The notion of a Gr\"obner-Shirshov basis is given in terms of compositions and triviality of compositions, encoding the notion of critical pairs in a rewriting system~\mcite{BN}.

\begin{defn}
Let $f, g \in \bfk\mathfrak{M}(\Omega,\,X)$ be $\Omega$-bracketed polynomials monic with respect to $\leq$. Let $\bar{f}$ be the leading monomial of $f$, and let $|f|$ denote its breadth.
\begin{enumerate}
\item  If there exist $u, v, w \in
\mathfrak{M}(\Omega,\,X)$ such that $w=\bar{f}u=v\bar{g}$ with $\max \big\{|\bar{f}|, |\bar{g}|\big\} < w<|\bar{f}|+|\bar{g}|$, then the $\Omega$-bracketed polynomial
$$(f, g)_{w}:=(f,g)_{u,v,w}:=fu-vg$$
is called the {\bf intersection composition of $f$ and $g$ with respect to} $(u,v)$.
\item If there exist $q \in \mathfrak{M}^{\star}(\Omega,\,X)$ and $w \in \mathfrak{M}(\Omega,\,X)$ such  that
$w=\bar{f}:=(f,g)_{q,w}:=q|_{\bar{g}}$,  then the $\Omega$-bracketed polynomial
$$(f, g)_{w}:=f-q|_{g}$$
is called the {\bf including composition of $f$ and $g$ with respect to} $q$.
\end{enumerate}
In both cases, the bracketed word $w$ is called the {\bf ambiguity} for the compositions.
\end{defn}

Now we arrive at the key notion of a Gr\"obner-Shirshov basis in which the confluncy of critical pairs is captured by a triviality condition.

\begin{defn} Let $S\subseteq \bfk\mathfrak{M}(\Omega,\,X)$
be a set of $\Omega$-bracketed polynomials that is monic with respect to a monomial order $\leq$, and let $w \in \mathfrak{M}(\Omega,\,X).$
\begin{enumerate}
\item
An element $u$ in $\bfk\frakM(X,\Omega)$ is called {\bf trivial modulo} $(S,w)$ if $u$ can be written as a linear combination $\sum_{i} c_{i}q_{i}|_{s_{i}}$ with $0\neq c_{i}\in \bfk, q_{i} \in \mathfrak{M}(\Omega,\,X)^{\star}, s_{i}\in S$ and $q_{i}|_{\bar{s_{i}}}<w.$ Then we denote
$$ u \equiv 0 \text{ mod }\, (S,W).$$
\item
The set $S$ is called a {\bf Gr\"{o}bner-Shirshov bases} (with respect to $\leq$), if for each pair $f, g \in S$ with $f\neq g$, every intersection composition and including composition $(f,g)_{w}$ of $f$ and $g$ is trivial modulo $(S, w).$
\item For $u, v\in \bfk\mathfrak{M}(\Omega,\,X),$ we say $u$ and $v$ are {\bf congruent modulo} $(S, w)$ and denote by
    $u\equiv v\, \text{ mod }\, (S, w)$
    if $u-v$ is trivial modulo $(S,w)$.
\end{enumerate}
\end{defn}

The following theorem is the {\bf Composition-Diamond Lemma for $\Omega$-(unitary) algebras}, adapting from the case for $\Omega$-nonunitary algebras in~\mcite{BCQ10}. See also~\mcite{QC18}.

\begin{theorem}\cite[Theorem~3.13]{GSZ}
Let $X$ be a set and $\Omega$ a nonempty set, and let $\leq$ be a monomial order on $\mathfrak{M}(\Omega,\,X)$. Let $S$ be a set of $\Omega$-bracketed polynomials in $\bfk\mathfrak{M}(\Omega,\,X)$ which are monic with respect to $\leq$ and let $\Id(S)$ be the $\Omega$-operated ideal of $\bfk\mathfrak{M}(\Omega,\,X)$ generated by $S$.
Then the following  statements are equivalent:
\begin{enumerate}
\item $S$ is a Gr\"{o}bner-Shirshov basis in $\bfk\mathfrak{M}(\Omega,\,X).$
\item For every non-zero $f \in \Id(S)$, then $\bar{f}=q |_{\bar{s}}$ for some $q \in \mathfrak{M}(\Omega,\,X)^{\star}$ and $ s\in S.$
\item Let
$$\Irr(S):=\ \left\{w\in \mathfrak{M}(\Omega,\,X)| w \neq q|_{\bar{s}}, q\in \mathfrak{M}(\Omega,\,X)^{\star}, s\in S\right\}
=\ \frakM(X,\Omega)\backslash \{q|_{\bar{s}}\,|\,q\in \mathfrak{M}(\Omega,\,X)^{\star}, s\in S\}.
$$
Then there is a linear decomposition $\bfk\frakM(X,\Omega)=\Id(S)\oplus \Irr(S)$. Thus $\Irr(S)$ modulo $\Id(S)$ is a \bfk-linear basis of $\bfk\mathfrak{M}(\Omega,\,X)/\Id(S).$
\end{enumerate}
\mlabel{them:CD}
\end{theorem}

\subsubsection{Gr\"{o}bner-Shirshov bases for free \mrbas}

We now show that the matching Rota-Baxter relations form a Gr\"{o}bner-Shirshov basis of the free $\Omega$-operated algebras $\bfk\mathfrak{M}(\Omega,\,X)$, and hence gives rise to a linear basis of the free \mrba thanks to the Composition-Diamond Lemma in Theorem~\mref{them:CD}.

Let $X$ and $\emptyset \neq \Omega$ be well-ordered sets. For notational convenience, we also denote $P_\omega(u)=\lc u\rc_\omega$. So one appearance of $P_\omega$ in a bracketed work $w\in \frakM(X,\Omega)$ means one appearance of a brackets $\lc \rc_\omega$ in $w$.
For $u=u_1\cdots u_r\in M(X)$ with $u_1,\cdots,u_r\in X$, define $\deg_X(u)=r$ if $u \ne 1$ and $\deg_X(1)=0$.
Extend the well order $\leq$ on $X$
to the {\bf degree lexicographical order} $\leq $  on $M(X)$ by taking, for  any $u=u_1\cdots u_r,v=v_1\cdots v_s \in M(X)\backslash\{1\}$, where $u_1,\cdots,u_r,v_1,\cdots, v_s\in X$,
\begin{equation}
u\leq v\Leftrightarrow
\left\{
\begin{array}{l}
\deg_X(u)<\deg_X(v),\\[5pt]
\text{or}~ \deg_X(u)=\deg_X(v)(=r) ~ \text{and} ~ (u_1, \cdots, u_r) \leq  (v_1, \cdots, v_r) \text{ lexicographically},
\end{array}
\right.
\label{eq:ordermx}
\end{equation}
Here we use the convention that the empty word $1\leq  u$ for all $u\in M(X)$.
Then $\leq$ is a well order on $M(X)$~\cite{BN}.

Further we extend $\leq$ to $\frakM(X,\Omega)$. Applying Eq.~(\mref{eq:wfact}) and grouping adjacent letters in $X$ together, we find that every $u \in \frakM(X)$ may be uniquely written as a product in the form
\begin{equation}
u=u_0 P_{\alpha_1}( \mpu_1) u_1 P_{\alpha_2}(\mpu_2)u_2 \cdots P_{\alpha_r}(\mpu_r) u_r,
\mlabel{eq:u}
\end{equation}
where
$$u_0,\cdots,u_r \in M(X),\, \mpu_1,\cdots, \mpu_r \in \frakM_{n-1}(X)\,\text{ and }\, {\alpha_1}, \cdots, {\alpha_r}\in \Omega.$$
Denote by $\deg_P(u)$ the number of occurrence of $P_\omega = \lfloor\ \rfloor_\omega, \omega\in \Omega$,
and define the $P$-breadth $\brep(u)$ of $u$ to be $r$.  For example, we have
$$u: = x_0 P_{\alpha_1}(x_1) x_2 P_{\alpha_2}(x_3P_{\alpha_3}(x_4)) x_5x_6 = u_0P_{\alpha_1}(\mpu_1)u_1P_{\alpha_2}(\mpu_2)u_2, \quad x_0, \cdots, x_6\in X, \quad \alpha_1,\alpha_2,\alpha_3\in \Omega,$$
where $u_0=x_0, u_1=x_2, u_3=x_5x_6, \mpu_1=x_1, \mpu_2=x_3P_{\alpha_3}(x_4), \deg_P(u) = 3$ and $\brep(u) = 2$.

Let $u,v\in \frakM(X)$ and write them uniquely in the form of Eq.~(\mref{eq:u}):
$$u = u_0 P_{\alpha_1}(\mpu_1) u_1P_{\alpha_2}(\mpu_2)u_2 \cdots P_{\alpha_r}(\mpu_r) u_r\,\text{ and }\,
v=v_0 P_{\beta_1}(\mpv_1) v_1 P_{\beta_2}(\mpv_2)v_2 \cdots P_{\beta_s}(\mpv_s) v_s.$$
We define $u\leqo v $ by induction on $\dep(u) + \dep(v)\geq 0$. For the initial step of $\dep(u) + \dep(v) = 0$, we have $u,v\in M(X)$ and use
the degree lexicographical order given in Eq.~(\mref{eq:ordermx}). For the induction step of $\dep(u) + \dep(v) \geq 1$, we define
\begin{equation*}
u\leqo v \Leftrightarrow
\left\{\begin{array}{ll}
\deg_P(u) < \deg_P(v), \\
\text{ or } \deg_P(u) = \deg_P(v)\, \text{ and }\, \brep(u) < \brep(v),\\
\text{ or } \deg_P(u) = \deg_P(v), \, \brep(u) = \brep(v) (=r)\, \text{ and }\\
 \,(P_{\alpha_1}, \mpu_1,\cdots,P_{\alpha_r}, \mpu_r,u_0,\cdots,u_r) \leq (P_{\beta_1}, \mpv_1,\cdots, P_{\beta_r}, \mpv_r, v_0,\cdots,v_r) \text{ lexicographically.}
\end{array}
\right.
\end{equation*}
Here $P_{\alpha_i}\leq P_{\beta_i}$ is compared by the order on $\Omega$ and
$\mpu_i\leqo \mpv_i$ and $u_i\leqo v_i$ are compared by the induction hypothesis.
With a similar argument to the case of $\leqo$ on $\frakM(X)$~\cite{ZGGS}, the above
defined $\leqo$ is a monomial order on $\frakM(X,\Omega)$.
In fact when $\Omega$ is a singleton, the above defined $\leqo$ is exactly the one given in ~\cite{ZGGS} on $\frakM(X)$. See also~\mcite{QC18}.

\begin{theorem}
With the order $\leq_{\mathrm{db}}$ on $\frakM(\Omega,\,X)$, the set
$$S=\left\{\left .\lf x\rf_\alpha\lf y\rf_\beta-\lf x\lf y\rf_\beta\rf_{\alpha}-\lf\lf x\rf_\alpha y\rf_{\beta}-\lambda_\beta\lf xy\rf_{\alpha}\,\right |\,x,y\in \frakM(\Omega,\,X), \alpha, \beta \in \Omega\right\}$$
is a Gr\"{o}bner-Shirshov basis in $\bfk\mathfrak{M}(\Omega,\,X)$.
\mlabel{thm:GS}
\end{theorem}

\begin{proof}
With the leading terms from $S$ in the form of $\lc x\rc_\alpha \lc y\rc_\beta$, all the possible ambiguities for compositions of $\Omega$-bracketed polynomials in $S$ are of the following three forms.
$$w_1:=\lf x\rf_\alpha\lf y\rf_\beta \lf z\rf_\gamma,\, w_2:=\lf u|_{\lf x\rf_\beta \lf y\rf_\gamma}\rf_\alpha
\lf z\rf_\delta,\, w_3:=\lf z\rf_\delta\lf u|_{\lf x\rf_\beta\lf y\rf_\gamma}\rf_\alpha.$$
where $x, y, z\in \mathfrak{M}(\Omega,\,X)$, $\alpha,\, \beta,\, \gamma,\delta \in \Omega$,  $u \in \mathfrak{M}(\Omega,\,X)^{\star}$. We now check that all these compositions are trivial.

\noindent{\bf Case 1.} $w_1=\lf x\rf_\alpha\lf y\rf_\beta \lf z\rf_\gamma$. In this case, we may write
\begin{align*}
f:=&\ f_{\alpha,\, \beta}(x, y)=\lf x \rf_\alpha\lf y\rf_\beta-\lf x\lf y\rf_\beta\rf_{\alpha}-\lf\lf x\rf_\alpha y\rf_{\beta}-\lambda_\beta \lf xy\rf_{\alpha},\\
g:=&\ g_{\beta,\, \gamma}(y, z) =\lf y\rf_\beta\lf z\rf_\gamma-\lf y\lf z\rf_\gamma\rf_{\beta}-\lf \lf y\rf_\beta z\rf_{\gamma}-\lambda_\gamma \lf yz \rf_{\beta}.
\end{align*}
Then we have
$$\bar{f}=\lf x\rf_\alpha \lf y\rf_\beta\, \text { and } \, \bar{g}=\lf y\rf_\beta\lf z\rf_\gamma.$$
Thus
\begin{align*}
(f, g)_{w_1}=& \ f\lf z\rf_\gamma-\lf x\rf_\alpha g\\
=& \ \lf x\rf_\alpha \lf y\rf_\beta \lf z\rf_\gamma
-\lf x\lf y\rf_\beta\rf_{\alpha}\lf z\rf_\gamma
-\lf \lf x\rf_\alpha y \rf_{\beta}\lf z \rf_\gamma-\lambda_\beta\lf xy \rf_{\alpha}\lf z\rf_\gamma \\
&\ - \lf x \rf_\alpha\lf y \rf_\beta \lf z\rf_\gamma
+\lf x\rf_\alpha \lf y \lf z\rf_\gamma\rf_{\beta}
+ \lf x\rf_\alpha \lf \lf y\rf_\beta z\rf_{\gamma}
+\lambda_\gamma \lf x\rf_\alpha \lf yz\rf_{\beta}\\
=& \ -\lf x\lf y\rf_\beta\rf_{\alpha}\lf z\rf_\gamma
-\lf \lf x\rf_\alpha y \rf_{\beta}\lf z \rf_\gamma
-\lambda_\beta\lf xy \rf_{\alpha}\lf z\rf_\gamma
+\lf x\rf_\alpha \lf y \lf z\rf_\gamma\rf_{\beta} \\
&\ + \lf x\rf_\alpha \lf \lf y\rf_\beta z\rf_{\gamma}
+\lambda_\gamma \lf x\rf_\alpha \lf yz\rf_{\beta}\\
=&\ -f_{\alpha,\, \gamma}(x\lf y\rf_\beta, z)
- \lf x \lf y\rf_\beta \lf z\rf_\gamma\rf_{\alpha}
-\lf \lf x\lf y\rf_\beta\rf_{\alpha} z\rf_{\gamma}
-\lambda_\gamma \lf x\lf y\rf_\beta z\rf_{\alpha}\\
& \ -g_{\beta,\, \gamma}(\lf x \rf_\alpha y, z)
-\lf \lf x\rf_\alpha y\lf z\rf_\gamma\rf_{\beta}
-\lf \lf \lf x \rf_\alpha y\rf_{\beta}z\rf_{\gamma}
-\lambda_\gamma \lf \lf x\rf_\alpha yz \rf_{\beta}\\
&\ - \lambda_\beta f_{\alpha, \gamma}(xy, z)
-\lambda_\beta \lf xy\lf z \rf_\gamma \rf_{\alpha}
-\lambda_\beta \lf \lf xy\rf_{\alpha}z\rf_{\gamma}
-\lambda_\beta \lambda_\gamma \lf xyz \rf_{\alpha}\\
& \ +f_{\alpha,\, \beta}(x, y\lf z\rf_\gamma)
+\lf x\lf y\lf z\rf_\gamma\rf_{\beta}\rf_{\alpha}
+\lf \lf x\rf_\alpha y\lf z\rf_\gamma\rf_{\beta}
+\lambda_\beta \lf xy \lf z\rf_\gamma\rf_{\alpha}\\
&\ + f_{\alpha,\, \gamma}(x, \lf y\rf_\beta z)
+\lf x \lf \lf y\rf_\beta z\rf_{\gamma}\rf_{\alpha}
+\lf \lf x\rf_\alpha \lf y \rf_\beta z\rf_{\gamma}
 +\lambda_\beta \lf x \lf y\rf_\beta z\rf_{\alpha}\\
& \ +\lambda_\gamma f_{\alpha,\, \beta}(x, yz)
+\lambda_\gamma \lf x\lf yz \rf_{\beta}\rf_{\alpha}
+\lambda_\gamma \lf \lf x\rf_\alpha yz\rf_{\beta}
+\lambda_\gamma \lambda_\beta \lf xyz\rf_{\alpha}\\
%
=&\ - \lf x \lf y\rf_\beta \lf z\rf_\gamma\rf_{\alpha}
+\lf x\lf y\lf z\rf_\gamma\rf_{\beta}\rf_{\alpha}
+\lf x \lf \lf y\rf_\beta z\rf_{\gamma}\rf_{\alpha}
+\lambda_\gamma \lf x\lf yz \rf_{\beta}\rf_{\alpha}\\
& \ +\lf \lf x\rf_\alpha \lf y \rf_\beta z\rf_{\gamma}
-\lf \lf x\lf y\rf_\beta\rf_{\alpha} z\rf_{\gamma}
-\lf \lf \lf x \rf_\alpha y\rf_{\beta}z\rf_{\gamma}
-\lambda_\beta \lf \lf xy\rf_{\alpha}z\rf_{\gamma}\\
&\ -f_{\alpha,\, \gamma}(x\lf y\rf_\beta, z)-g_{\beta,\, \gamma}(\lf x \rf_\alpha y, z)-\lambda_\beta f_{\alpha, \gamma}(xy, z)+f_{\alpha,\, \beta}(x, y\lf z\rf_\gamma)
\\
& \
+ f_{\alpha,\, \gamma}(x, \lf y\rf_\beta z)
+\lambda_\gamma f_{\alpha,\, \beta}(x, yz)
\\
=& \  -\lf x\, g_{\beta,\, \gamma}(y,z) \rf_\alpha
+\lf f_{\alpha,\, \beta}(x, y)z \rf_\gamma
-f_{\alpha,\, \gamma}(x\lf y\rf_\beta, z)
-g_{\beta,\, \gamma}(\lf x \rf_\alpha y, z)\\
&\ - \lambda_\beta f_{\alpha, \gamma}(xy, z)
+f_{\alpha,\, \beta}(x, y\lf z\rf_\gamma)
+ f_{\alpha,\, \gamma}(x, \lf y\rf_\beta z)
+\lambda_\gamma f_{\alpha,\, \beta}(x, yz)\\
=& \ -\lf x \star|_{g_{\beta,\, \gamma}(y,z)} \rf_\alpha
+\lf \star|_{f_{\alpha,\, \beta}(x, y)}\, z \rf_\gamma
-\star|_{f_{\alpha,\, \gamma}(x\lf y\rf_\beta, z)}
-\star|_{g_{\beta,\, \gamma}(\lf x \rf_\alpha y, z)}\\
&\ -\lambda_\beta \star |_{ f_{\alpha, \gamma}(xy, z)}
+\star|_{f_{\alpha,\, \beta}(x, y\lf z\rf_\gamma)}
+\star|_{f_{\alpha,\, \gamma}(x, \lf y\rf_\beta z)}
+\lambda_\gamma \star|_{ f_{\alpha,\, \beta}(x, yz)},
\end{align*}
which is trivial modulo  $(S,w_1)$ since
 \begin{align*}
\lf x\, \star|_{\lbar {g_{\beta,\, \gamma}(y,z)}} \rf_\alpha=&\ \lf x \star|_{\lf y\rf_\beta \lf z\rf_\gamma} \rf_\alpha=\lf x \lf y\rf_\beta \lf z\rf_\gamma\rf_{\alpha}<_{{\rm bd}} \lf x\rf_\alpha\lf y\rf_\beta \lf z\rf_\gamma=w_1\\
\lf \star|_{\lbar{f_{\alpha,\, \beta}(x, y)}}\, z \rf_\gamma=&\ \lf  \star|_{\lf x\rf_\alpha \lf y\rf_\beta} z\rf_\gamma=\lf  \lf x\rf_\alpha \lf y\rf_\beta z\rf_{\gamma}<_{{\rm bd}} \lf x\rf_\alpha\lf y\rf_\beta \lf z\rf_\gamma=w_1,\\
\star|_{\lbar{f_{\alpha,\, \gamma}(x\lf y\rf_\beta, z)}}=&\ \star|_{\lf x\lf y\rf_\beta\rf_{\alpha}\lf z\rf_\gamma}=\lf x\lf y\rf_\beta\rf_{\alpha}\lf z\rf_\gamma
<_{{\rm bd}} \lf x\rf_\alpha\lf y\rf_\beta \lf z\rf_\gamma=w_1,\\
\star|_{\lbar{g_{\beta,\, \gamma}(\lf x \rf_\alpha y, z)}}=&\ \star|_{\lf \lf x \rf_\alpha y\rf_{\beta}\lf z\rf_\gamma}=\lf \lf x \rf_\alpha y\rf_{\beta}\lf z\rf_\gamma<_{{\rm bd}} \lf x\rf_\alpha\lf y\rf_\beta \lf z\rf_\gamma=w_1,\\
\star|_{\lbar{ f_{\alpha, \gamma}(xy, z)}}=&\ \star|_{\lf xy \rf_{\alpha}\lf z\rf_\gamma}=\lf xy \rf_{\alpha}\lf z\rf_\gamma<_{{\rm bd}} \lf x\rf_\alpha\lf y\rf_\beta \lf z\rf_\gamma=w_1,\\
\star|_{\lbar{f_{\alpha,\, \beta}(x, y\lf z\rf_\gamma)}}=&\ \star|_{\lf x\rf_\alpha \lf y\lf z\rf_\gamma\rf_{\beta}}=\lf x\rf_\alpha \lf y\lf z\rf_\gamma\rf_{\beta}<_{{\rm bd}} \lf x\rf_\alpha\lf y\rf_\beta \lf z\rf_\gamma=w_1,\\
\star|_{\lbar{f_{\alpha,\, \gamma}(x, \lf y\rf_\beta z)}}=&\ \star|_{\lf x\rf_\alpha \lf \lf y\rf_\beta z\rf_{\gamma}}=\lf x\rf_\alpha \lf \lf y\rf_\beta z\rf_{\gamma}<_{{\rm bd}} \lf x\rf_\alpha\lf y\rf_\beta \lf z\rf_\gamma=w_1,\\
\star|_{\lbar { f_{\alpha,\, \beta}(x, yz)}}=&\ \star|_{\lf x\rf_\alpha\lf yz \rf_{\beta}}=\lf x\rf_\alpha\lf yz \rf_{\beta}<_{{\rm bd}} \lf x\rf_\alpha\lf y\rf_\beta \lf z\rf_\gamma=w_1.
\end{align*}

\noindent{\bf Case 2.} $w_2=\lf u|_{\lf x\rf_\beta \lf y\rf_\gamma}\rf_\alpha\lf z\rf_\delta$. In this case, we may write
\begin{align*}
q:=\ &\lf u \rf_\alpha \lf z\rf_\delta,\,
 g:=\, g_{\beta,\, \gamma}(x, y) =\lf x\rf_\beta\lf y\rf_\gamma-\lf x\lf y\rf_\gamma\rf_{\beta}-\lf \lf x\rf_\beta y\rf_{\gamma}
-\lambda_\gamma \lf xy \rf_{\beta},\\
f:=\ & f_{\alpha,\, \delta}(u|_{\lf x\rf_\beta \lf y\rf_\gamma},\,z)=\lf u|_{\lf x\rf_\beta \lf y\rf_\gamma} \rf_\alpha\lf z\rf_\delta-\lf u|_{\lf x\rf_\beta \lf y\rf_\gamma}\lf z\rf_\delta\rf_{\alpha}-\lf\lf u|_{\lf x\rf_\beta \lf y\rf_\gamma}\rf_\alpha z\rf_{\delta}
-\lambda_\delta \lf u|_{\lf x\rf_\beta \lf y\rf_\gamma}z\rf_{\alpha}.
\end{align*}
Then we have
$$\bar{f}=\lf u|_{\lf x\rf_\beta \lf y\rf_\gamma}\rf_\alpha\lf z\rf_\delta, \, \bar{g}=\lf x\rf_\beta\lf y\rf_\gamma,\, \text{ and } \, \bar{f}=q|_{\bar{g}}.$$
Thus
\begin{align*}
(f, g)_{\omega_2}=& \ f-q|_g\\
=& \ \lf u|_{\lf x\rf_\beta\lf y\rf_\gamma}\rf_\alpha \lf z\rf_\delta
-\lf u|_{\lf x\rf_\beta\lf y\rf_\gamma}\lf z\rf_\delta\rf_{\alpha}
-\lf \lf u|_{\lf x\rf_\beta\lf y\rf_\gamma} \rf_\alpha z\rf_{\delta}
-\lambda_\delta \lf u|_{\lf x\rf_\beta \lf y\rf_\gamma}z\rf_{\alpha}\\
& \ -\lf u|_{\lf x\rf_\beta\lf y\rf_\gamma}\rf_\alpha \lf z\rf_\delta
+\lf u|_{\lf x\lf y\rf_\gamma\rf_{\beta}}\rf_\alpha \lf z\rf_\delta
+\lf u|_{\lf \lf x\rf_\beta y\rf_{\gamma}}\rf_\alpha\lf z\rf_\delta
+\lambda_\gamma \lf u|_{\lf xy\rf_{\beta}}\rf_\alpha\lf z\rf_\delta\\
=& \ -\lf u|_{\lf x\rf_\beta\lf y\rf_\gamma}\lf z\rf_\delta\rf_{\alpha}
-\lf \lf u|_{\lf x\rf_\beta\lf y\rf_\gamma} \rf_\alpha z\rf_{\delta}
-\lambda_\delta \lf u|_{\lf x\rf_\beta \lf y\rf_\gamma}z\rf_{\alpha}
 +\lf u|_{\lf x\lf y\rf_\gamma\rf_{\beta}}\rf_\alpha \lf z\rf_\delta\\
&\ +\lf u|_{\lf \lf x\rf_\beta y\rf_{\gamma}}\rf_\alpha\lf z\rf_\delta
+\lambda_\gamma \lf u|_{\lf xy\rf_{\beta}}\rf_\alpha\lf z\rf_\delta\\
=& \
-\lf u|_{g_{\beta,\, \gamma}(x, y)}\lf z\rf_\delta\rf_{\alpha}
- \lf u|_{\lf x\lf y\rf_\gamma\rf_{\beta}}\lf z\rf_\delta\rf_{\alpha}
-\lf u|_{\lf \lf x\rf_\beta y\rf_{\gamma}}\lf z\rf_\delta\rf_{\alpha}
- \lambda_\gamma \lf u|_{\lf xy\rf_{\beta}}\lf z\rf_\delta\rf_{\alpha}\\
& \
-\lf \lf u|_{g_{\beta,\, \gamma}(x, y)}\rf_\alpha z\rf_{\delta}
- \lf \lf u|_{\lf x\lf y\rf_\gamma\rf_{\beta}}\rf_\alpha z\rf_{\delta}
- \lf \lf u|_{\lf \lf x \rf_\beta y\rf_{\gamma}}\rf_\alpha z\rf_{\delta}
- \lambda_\gamma \lf \lf u|_{\lf xy\rf_{\beta}}\rf_\alpha z\rf_{\delta}\\
&\
-\lambda_\delta \lf u|_{g_{\beta,\, \gamma}(x, y)}z\rf_{\alpha}
-\lambda_\delta \lf u|_{\lf x\lf y\rf_\gamma\rf_{\beta}}z\rf_{\alpha}
-\lambda_\delta \lf u|_{\lf \lf x \rf_\beta y\rf_{\gamma}}z\rf_{\alpha}
-\lambda_\delta \lambda_\gamma \lf u|_{\lf xy\rf_\beta}z\rf_{\alpha} \\
&\
+ f_{\alpha,\, \delta}(u|_{\lf x\lf y\rf_\gamma\rf_{\beta}}, z)
+ \lf u|_{\lf x\lf y\rf_\gamma\rf_{\beta}}\lf z\rf_\delta\rf_{\alpha}
+ \lf \lf u|_{\lf x\lf y\rf_\gamma\rf_{\beta}}\rf_\alpha z\rf_{\delta}
 + \lambda_\delta \lf u|_{\lf x\lf y\rf_\gamma\rf_{\beta}}z\rf_{\alpha} \\
 &\
+ f_{\alpha,\,\delta}(u|_{\lf \lf x\rf_\beta y\rf_{\gamma}}, z)
+\lf u|_{\lf \lf x\rf_\beta y\rf_{\gamma}}\lf z\rf_\delta\rf_{\alpha}
+ \lf \lf u|_{\lf \lf x \rf_\beta y \rf_{\gamma}}\rf_\alpha z\rf_{\delta}
+\lambda_\delta \lf u|_{\lf \lf x\rf_\beta y\rf_{\gamma}}z\rf_{\alpha} \\
& \
 + \lambda_\gamma f_{\alpha,\, \delta}(u|_{\lf xy\rf_{\beta}}, z)
 +\lambda_\gamma \lf u|_{\lf xy\rf_{\beta}} \lf z\rf_\delta\rf_{\alpha}
 + \lambda_\gamma \lf \lf u|_{\lf xy\rf_{\beta}}\rf_\alpha z\rf_{\delta}
 + \lambda_\gamma\lambda_\delta \lf u|_{\lf xy\rf_{\beta}}z\rf_{\alpha}\\
=& \ -\lf u|_{g_{\beta,\, \gamma}(x, y)}\lf z\rf_\delta\rf_{\alpha}
-\lf \lf u|_{g_{\beta,\, \gamma}(x, y)}\rf_\alpha z \rf_{\delta}
-\lambda_\delta \lf u|_{g_{\beta,\, \gamma}(x, y)} z\rf_{\alpha}
+ f_{\alpha,\, \delta}(u|_{\lf x\lf y\rf_\gamma\rf_{\beta}},\, z)\\
&\ + f_{\alpha,\, \delta}(u|_{\lf \lf x \rf_\beta y\rf_{\gamma}},\, z)
 + \lambda_\gamma f_{\alpha,\, \delta}(u|_{\lf xy\rf_{\beta}},\, z)\\
=& \ -\star|_{\lf u|_{g_{\beta,\, \gamma}(x, y)}\lf z\rf_\delta\rf_{\alpha}}
-\star|_{\lf u|_{g_{\beta,\, \gamma}(x, y)}\rf_\alpha z \rf_{\delta}}
-\lambda_\delta\star|_{ \lf u|_{g_{\beta,\, \gamma}(x, y)} z\rf_{\alpha}}
+\star|_{f_{\alpha,\, \delta}(u|_{\lf x\lf y\rf_\gamma\rf_{\beta}},\, z)}\\
& \ +\star|_{f_{\alpha,\, \delta}(u|_{\lf \lf x \rf_\beta y\rf_{\gamma}}, z)}
+\lambda_\gamma \star|_{ f_{\alpha,\, \delta}(u|_{\lf xy\rf_{\beta}},\, z)},
\end{align*}
which is trivial modulo $(S, w_2)$ since
\begin{align*}
\star|_{\lbar{\lf u|_{f_{\beta,\, \gamma}(x, y)}\lf z\rf_\delta\rf_{\alpha}}}&= \star|_{\lf u|_{\lf x\rf_\beta\lf y\rf_\gamma}\lf z \rf_\delta\rf_{\alpha}}
=\lf u|_{\lf x\rf_\beta\lf y\rf_\gamma}\lf z \rf_\delta\rf_{\alpha}<_{{\rm bd}} \lf u|_{\lf x\rf_\beta \lf y\rf_\gamma}\rf_\alpha
\lf z\rf_\delta=w_2,\\
\star|_{\lbar{\lf u|_{f_{\beta,\, \gamma}(x, y)}\rf_\alpha z \rf_{\delta}}}&=\star|_{\lf u|_{\lf x\rf_\beta\lf y\rf_\gamma} z\rf_{\delta}}=\lf u|_{\lf x\rf_\beta\lf y\rf_\gamma} z\rf_{\delta}<_{{\rm bd}} \lf u|_{\lf x\rf_\beta \lf y\rf_\gamma}\rf_\alpha
\lf z\rf_\delta
=w_2, \\
\star|_{\lbar{ \lf u|_{f_{\beta,\, \gamma}(x, y)} z\rf_{\alpha}}}&=\star|_{\lf u|_{\lf x \rf_\beta\lf y\rf_\gamma} z\rf_{\alpha}}= \lf u|_{\lf x \rf_\beta\lf y\rf_\gamma} z\rf_{\alpha}
<_{{\rm bd}} \lf u|_{\lf x\rf_\beta \lf y\rf_\gamma}\rf_\alpha
\lf z\rf_\delta=w_2,\\
\star|_{\lbar{f_{\alpha,\, \delta}(u|_{\lf x\lf y\rf_\gamma\rf_{\beta}},\, z)}}&=
\star|_{\lf u|_{\lf x\lf y\rf_\gamma\rf_{\beta}}\rf_\alpha \lf z \rf_\delta}= \lf u|_{\lf x\lf y\rf_\gamma\rf_{\beta}}\rf_\alpha \lf z \rf_\delta <_{{\rm bd}} \lf u|_{\lf x\rf_\beta \lf y\rf_\gamma}\rf_\alpha
\lf z\rf_\delta=w_2, \\
\star|_{\lbar{f_{\alpha,\, \delta}(u|_{\lf \lf x \rf_\beta y\rf_{\gamma}}, z)}}&=\star|_{\lf u|_{\lf \lf x\rf_\beta y\rf_{\gamma}}\rf_\alpha \lf z\rf_\delta}=\lf u|_{\lf \lf x\rf_\beta y\rf_{\gamma}}\rf_\alpha \lf z\rf_\delta<_{{\rm bd}} \lf u|_{\lf x\rf_\beta \lf y\rf_\gamma}\rf_\alpha
\lf z\rf_\delta=w_2,\\
\star|_{\lbar{ f_{\alpha,\, \delta}(u|_{\lf xy\rf_{\beta}},\, z)}}&= \star|_{\lf u|_{\lf xy \rf_{\beta}}\rf_\alpha \lf z\rf_\delta}
=\lf u|_{\lf xy \rf_{\beta}}\rf_\alpha \lf z\rf_\delta <_{{\rm bd}} \lf u|_{\lf x\rf_\beta \lf y\rf_\gamma}\rf_\alpha
\lf z\rf_\delta=w_2.
\end{align*}
\noindent{\bf Case 3.} $w_3=\lf z\rf_\delta\lf u|_{\lf x\rf_\beta\lf y\rf_\gamma}\rf_\alpha$. The proof is similar to that of Case 2.

This completes the proof.
\end{proof}

\subsection{Construction of free matching Rota-Baxter algebras on a set}
Now we apply Theorems~\mref{them:CD} and \mref{thm:GS} to construct free matching Rota-Baxter algebras on a set.

Let $X$ be a set. We first describe the set $\mathrm{Irr}(S)$, where
$$S=\left\{\left . \lf x\rf_\alpha\lf y\rf_\beta-\lf x\lf y\rf_\beta\rf_{\alpha}-\lf\lf x\rf_\alpha y\rf_{\beta}-\lambda_\beta\lf xy\rf_{\alpha}\,\right |\, x,y\in \frakM(\Omega,\,X), \alpha, \beta \in \Omega\right\}.$$
The set $\mathrm{Irr}(S)$ will give a linear basis of the free matching Rota-Baxter algebras on the set $X$.

By Theorems~\mref{them:CD} and \mref{thm:GS}, the set consists of bracketed words in $\frakM(X,\Omega)$ that {\em do not} contain subwords of the form $\lc u\rc_\alpha \lc v\rc_\beta$ for any $u, v\in \frakM(X,\Omega)$ and $\alpha,\beta\in \Omega$. As in the case of one operator~\mcite{Gub}, we give a description of $\Irr(S)$ by inclusion conditions, rather than the above exclusive conditions.
For subsets $U, V\subseteq \frakM(X,\Omega)$ and $r\geq 1$, we use the abbreviations
$$ UV:=\{uv\,|\, u\in U, v\in V\}, \quad U^r:=\{u_1\cdots u_r\,|\, u_i\in U, 1\leq i\leq r\},\quad \lc U\rc_\Omega:=\{\lc u\rc_\omega\,|\, u\in U, \omega\in \Omega\}.$$

\begin{defn}
Let $Y,Z$ be subsets of $\mathfrak{M}(\Omega,\,X)$. Define the {\bf alternating products} of $Y$ and $Z$  by
\begin{align}
\Lambda(Y,Z):&=
\left(\bigcup_{r\geq1}(Y\lfloor Z\rfloor_\Omega)^{r}\right)\bigcup
\left(\bigcup_{r\geq0}(Y\lfloor Z\rfloor_\Omega)^rY\right)\,\bigcup
\left(\bigcup_{r\geq1}(\lfloor Z\rfloor_\Omega Y)^{r}\right)\bigcup
\left(\bigcup_{r\geq0}(\lfloor Z\rfloor_\Omega Y)^r\lfloor Z\rfloor_\Omega\right)\bigcup
\left\{1\right\},\nonumber
\label{eq:alternating}
\end{align}
where $1$ is the identity in $\mathfrak{M}(\Omega,\,X)$.
\end{defn}

We observe that $\Lambda(Y,Z)\subseteq\mathfrak{M}(\Omega,\,X)$. Then we recursively define
$$\frak X_{0}:=M(X)=S(X)\cup\{1\}\,\text{ and }\,\frak X_{n}:=\Lambda(S(X),\frak X_{n-1}),\, n\geq1.$$
Thus $\frak X_{0}\subseteq\cdots\subseteq\frak X_{n}\subseteq\cdots.$
Finally we define $$\frak X_{\infty}:=\dirlim\frak X_{n} =\bigcup_{n\geq 0} \frak X_{n}.$$
Elements in $\frak X_{\infty}$
are called {\bf \match Rota-Baxter words} (MRBWs).
For a MRBW $w\in\frak X_{\infty}$, we call $\dep(w):=\min\{n\mid w\in \frak X_{n}\}$ the {\bf depth} of $w$, which agrees with the depth of $w$ as an element of $\frakM(X,\Omega)$ in Section~\mref{ss:freeoa}.

The following properties of MRBWs are easy to verify as in the case of one operator~\mcite{Gub}.

\begin{lemma}\mlabel{lem:wprod}
Every MRBW $w\neq 1$ has a unique {\bf alternating decomposition:} $w=w_1\cdots w_m$, where $w_i\in X\cup \lfloor \frakX_\infty\rfloor_{\Omega} $, $1\leq i\leq m$, $m\geq 1,$ and no consecutive elements in the sequence $w_1,\cdots, w_m$ are in $\lfloor \frakX_\infty\rfloor_{\Omega}$.
\end{lemma}

Recall that $\bfk\mathfrak{M}(\Omega,\,X)/\Id(S)$ is a free \mrba on the set $X$, and has a linear basis $\Irr(S)$ by Theorems~\mref{them:CD} and~\mref{thm:GS}. Thus with the evident identification of $\Irr(S)$ with $\frakX_\infty$ thanks to Lemma~\mref{lem:wprod}, we obtain

\begin{theorem}
The set $\Irr(S)=\frak X_{\infty}$ modulo $\Id(S)$ is a linear basis of the free \mrba $\bfk\mathfrak{M}(\Omega,\,X)/\Id(S)$ of weight $\lambda_\Omega$ on $X$.
 \mlabel{them:FreeRB}
\end{theorem}

Thus the quotient map
$\varphi: \bfk\mathfrak{M}(\Omega,\,X) \to \bfk\mathfrak{M}(\Omega,\,X)/\Id(S)$
of $\Omega$-operated algebras becomes the projection
\begin{equation}
 \varphi: \bfk\mathfrak{M}(\Omega,\,X) = \Id(S)\oplus \bfk \frakX_\infty\longrightarrow \bfk \frakX_\infty.
\mlabel{eq:phidef}
\end{equation}

Let $\ncshaw(X,\Omega):=\bfk\frak X_{\infty}$ be the resulting free \mrba on $X$ with its matching Rota-Baxter operators $P_\omega, \omega\in \Omega$ and multiplication $\dw$.
We first note the inclusions $\lc \frakX_\infty \rc_\alpha \subseteq \frakX_\infty, \alpha\in \Omega$. Thus the matching Rota-Baxter operator $P_\omega$ on $\bfk\frakX_\infty$ is simply
\begin{equation}
P_\omega(w)=\lc w\rc_\omega.
\mlabel{eq:wop}
\end{equation}

For the product $\dw$ in $\ncshaw(X,\Omega)$, we have the following algorithm in analog to those for the product of two Rota-Baxter words (resp. Rota-Baxter system words) in the free Rota-Baxter algebra~\mcite{Gub} (resp. free Rota-Baxter system~\mcite{QC18}).
\begin{algorithm}
Let $w, w'\in \frak X_{\infty}$. We define the product $w\dw  w'$ inductively on the sum of depths $n:=\dep(w)+\dep(w')\geq 0$.
\begin{enumerate}
\item If $n=0$, then $w,w'\in\frak X_{0}=M(X)$ and define $w\dw  w':=ww'$.
\item Let $k\geq 0$. Suppose that $w\dw  w'$ have been defined for $0\leq n\leq k$ and consider
the case of $n=k+1$.
We need to consider the following two cases.

\noindent{\bf Case 1.}  $\bre(w),\bre(w')\leq 1$. We define
\begin{equation}
\begin{aligned}
w\dw  w'=\left\{\begin{array}{ll}
\lc \lbar{w}\dw  w'\rc_\alpha+\lc w\dw  \lbar{w}'\rc_\beta+\lambda_\beta \lc\lbar{w}\dw \lbar{w}'
\rc_\alpha,&\text{ if } w=\lc\lbar{w}\rc_\alpha \,\text{ and }\, w'=\lc\lbar{w}'\rc_\beta,\\
ww',&\text{otherwise},
\end{array}\right.
\end{aligned}
\mlabel{eq:Bdia}
\end{equation}
where $\alpha, \beta \in \Omega$.

\noindent{\bf Case 2.} $\bre(w)\geq2$ or $\bre(w')\geq2$. Let $w=w_{1}\cdots w_{m}$ and $w'=w'_{1}\cdots w'_{m'}$ be the alternating decompositions of $w$ and $w'$, respectively. Define
\begin{equation}
w\dw  w':=w_{1}\cdots w_{m-1} (w_{m}\dw  w'_{1})w'_{2}\cdots w'_{m'},
\mlabel{eq:cdiam}
\end{equation}
where $w_{m}\dw  w'_{1}$ is defined by Eq.~(\mref{eq:Bdia}) and the rest of the products are
given by the concatenation.
\end{enumerate}
\end{algorithm}

Different from the filtration defined by the depth, the grading $\frakM^{(n)}$ and filtration $\frakM_{(n)}$ on $\frakM(X,\Omega)$ defined by the total degree $\deg_{td}$ in Eq.~(\mref{eq:wdegfil}) restrict to those on $\frakX_\infty$:
\begin{equation}
\frakX^{(n)}:=\frakM^{(n)}\cap \frakX_\infty, \quad \frakX_{(n)}:=\frakM_{(n)}\cap \frakX_\infty, \quad n\geq 0.
\mlabel{eq:rbdegfil}
\end{equation}
The resulting grading $\ncshaw(X,\Omega)=\oplus_{n\geq 0} \bfk \frakX^{(n)}$ holds only linearly since the multiplication does not preserve the grading. However, we have the following compatibility for the filtered structures.

\begin{prop}\mlabel{pp:wfil}
The free \mrba $\ncshaw(X,\Omega)$, with the filtration $${\ncshaw}(X,\Omega)_{(n)}:=\bfk \frakX_{(n)}, n\geq 0,$$ is an $\Omega$-operated filtered algebra as defined in Definition~\mref{de:ofil}$:$
\begin{equation}
P_\alpha(\bfk \frakX_{(p)}) \subseteq \bfk \frakX_{(p+1)}, \quad (\bfk \frakX_{(p)}) \dw (\bfk \frakX_{(q)}) \subseteq \bfk \frakX_{(p+q)}, \quad p, q\geq 0.
\mlabel{eq:wfil}
\end{equation}
Moreover, the homomorphism $\varphi: \bfk\frakM(X,\Omega)\to \ncshaw(X,\Omega)$ of $\Omega$-operated algebras preserves the filtrations$:$
\begin{equation}
\varphi(\bfk\frakM(X,\Omega)_{(n)})\subseteq \bfk\frakX_{(n)}, n\geq 0.
\mlabel{eq:wfilhom}
\end{equation}
\end{prop}

\begin{proof}
The first inclusion in Eq.~\eqref{eq:wfil} follows from the definition of the operators $P_\alpha$ in Eq.~(\mref{eq:wop}).

For the second inclusion in Eq.~\eqref{eq:wfil}, by linearity we just need to prove
$$  \frakX_{(p)} \dw  \frakX_{(q)} \subseteq \bfk \frakX_{(p+q)}, \quad p, q\geq 0,$$
for which we apply induction on $p+q\geq 0$, with the general remark that $\deg_{td}$ is additive with respect to the concatenation product. Thus for $w\in \frakX_{(p)}$ and $w'\in \frakX_{(q)}$, we have $w \dw w' \in \frakX_{(p+q)}$ as long as $w\dw w'=ww'$ is the concatenation.

When $p+q=0$, we have $p=q=0$. Since $\frakX_{(0)}=\frakX^{(0)}=M(X)=\frakM_0$ on which the product $\dw$ is the concatenation, the inclusion holds by the above general remark. Let $k\geq 0$. Assume that the inclusion holds for $p+q\leq k$ and consider the case when $p+q=k+1$. If either $p$ or $q$ is zero, then $\frakX_p\dw \frakX_q$ is the concatenation and the desired inclusion again follows. If none of $p$ or $q$ is zero and consider $w=w_1\cdots w_m\in \frakX_p$ and $w'=w'_1\cdots w'_{m'}\in \frakX_q$ with their alternating decompositions. Then $w\dw w'$ is again the concatenation and $w\dw w'$ is in $\frakX_{(p+q)}$ except when $w_m=\lc \lbar{w}_m\rc_{\alpha}$ and $w'_1=\lc \lbar{w}'_1\rc_\beta$ in which case, Eq.~(\mref{eq:Bdia}) gives
\begin{align*}
 w\dw w'= &w_1\cdots w_{m-1}\lc \lbar{w}_1\dw  w'_{m'}\rc_\alpha w'_2\cdots w'_{m'}+w_1\cdots w_{m-1}\lc w\dw  \lbar{w}'\rc_\beta w'_2\cdots w'_{m'}\\
 &+\lambda_\beta w_1\cdots w_{m-1} \lc\lbar{w}\dw \lbar{w}'
\rc_\alpha w'_2\cdots w'_{m'}.
\end{align*}
Since all the products are the concatenation except the ones in the brackets, by the general remark again, we just need to show that each of the brackets is in $\frakX_{(\deg_{td}(w_m)+\deg_{dt}(w'_1))}$.
But by the induction hypothesis, the $\dw$-products inside the three brackets are in $\bfk \frakX_{(\deg_{td}(w_m)-1+\deg_{td}(w'_1))}$. Hence the three brackets are in $\bfk \frakX_{(\deg_{td}(w_m)+\deg_{td}(w'_1))}$ by the first inclusion in Eq.~(\mref{eq:wfil}). This completes the induction.

We finally prove Eq.~(\mref{eq:wfilhom}) by induction on $n\geq 0$. The initial case of $n=0$ holds since $\frakM_0=M(X)$ equals $\frakX_0$ on which $\varphi$ is the identity. For a given $k\geq 0$, assume that $\varphi(\bfk \frakM_n)\subseteq \bfk \frakX_n$ for $n\geq k$ and consider $1\neq w\in \frakM_{k+1}$. If the width of $w$ is one, then $w$ is either in $X$ or is of the form $\lc \lbar{w}\rc_\omega, \lbar{w}\in\frakM_\infty, \omega\in \Omega.$ The former case is already proved. For the latter case, we have $\lbar{w}\in \frakM_{(k)}$ and so $\varphi(\lbar{w})$ is in $\frakX_{(k)}$ by the induction hypothesis and then $\varphi(w)=P_w(\varphi(\lbar{w}))$ is in $\bfk\frakX_{(k+1)}$ by the first inclusion in Eq.~(\mref{eq:wfil}). If the width of $w$ is greater than one, then $w=w_1w_2$ with $w_1, w_2\in \frakM_{(k)}$. Thus by the induction hypothesis,
$\varphi(w_i)\in \bfk\frakX_{(\deg_{td}(w_i))}, i=1,2.$ Since $\varphi$ is an algebra homomorphism, by the second inclusion in Eq.~(\mref{eq:wfil}), we have
$$ \varphi(w)=\varphi(w_1)\dw \varphi(w_2) \in \bfk \frakX_{(\deg_{td}(w_1))} \dw \frakX_{(\deg_{td}(w_2))} \subseteq \bfk\frakX_{(\deg_{td}(w))}.$$
This completes the induction.
\end{proof}

\subsection{Free \mrbas on decorated rooted forests}
In view of the convenience of working with rooted forests for Hopf algebra structures in the next section, we apply the isomorphism $\theta:\bfk\mathfrak{M}(X,\,\Omega) \overset{\sim}{\rightarrow} {\bfk}\rfs$ in Eq.~(\mref{eq:Isomorphism2}) of  $\Omega$-operated algebras to reformulate the main results in this section in terms of rooted forests. We first write
\begin{equation}
\mathfrak{L}_{n}:=\theta(\mathfrak{X}_{n}),\,n\geq0 \,\text{ and }\,  \mathfrak{L}_{\infty}:= \dirlim \frakL_n = \dirlim \theta(\frakX_n) = \theta(\mathfrak{X}_{\infty}),
\mlabel{eq:lbasisn} \notag
\end{equation}
giving rise to linear isomorphism
\begin{equation}
\ncshaw(X,\Omega)=\bfk \frakX_\infty \overset{\theta}{\longrightarrow} \ncshall(X,\Omega):= {\bfk}\mathfrak{L}_{\infty}, \quad \Id(S) \overset{\theta}{\longrightarrow} \Id(\frakS)
\mlabel{eq:isolw}
\end{equation}
for the operated ideals generated by $S$ and $\frakS$ defined in Eqs.~(\mref{eq:gsw}) and (\mref{eq:gsf}) respectively.
Similar to \match Rota-Baxter bracketed words, elements in $\mathfrak{L}_{\infty}$ are called {\bf \match Rota-Baxter forests} (MRBFs).

Further we obtain a homomorphism of $\Omega$-operated algebras
\begin{equation}
\psi:=\theta \varphi \theta^{-1}:(\bfk \rfs ,\,\cdot,\, \Bo)\rightarrow (\ncshal(X,\Omega),\, \dl ,\, \Bo),
\mlabel{eq:alghom}
\end{equation}
yielding the commutative diagram
\begin{equation}
\begin{split}
\xymatrix{
\bfk \frakM(X,\Omega)=\Id(S) \oplus \ncshaw(X,\Omega) \ar^(.65){\varphi}[rr] \ar^{\theta}[d] && \ncshaw(X,\Omega) \ar^{\theta}[d]\\
\bfk \calf(X,\Omega)=\Id(\frakS) \oplus \ncshal(X,\Omega) \ar^(.65){\psi}[rr] && \ncshal(X,\Omega)
}
\end{split}
\mlabel{eq:freediag}
\end{equation}

The following result shows an elementary property of $\psi$.

\begin{lemma}
Let $i: \ncshall(X,\Omega)= {\bfk}\mathfrak{L}_{\infty} \to \bfk\rfs$ be the natural inclusion.
Then $\psi i = \id_{\ncshal(X,\Omega)}$. Consequently, $i\psi$ is idempotent.
\mlabel{lem:varid}
\end{lemma}

With this transporting of structures, the free \mrba structure on $\ncshaw(X,\Omega)$ gives rise to a free \mrba structure on $\ncshal(X,\Omega)$. More precisely,  define a product
$$\dl :\ncshal(X,\Omega) \ot \ncshal(X,\Omega)\rightarrow\ncshal(X,\Omega)$$ by taking
\begin{equation}
F\dl F':=\theta(\theta^{-1}(F)\dw \theta^{-1}(F'))~\text{ for } F, F'\in \ncshall(X,\Omega).
\mlabel{eq:lproduct}
\end{equation}
Also define a linear operator on $\ncshal(X,\Omega)$ by $\theta P_\omega \theta^{-1}$ which turns out to be just the grafting operator $B_\omega ^+$.

Moreover, the degree $\deg$ on $\calf(X,\Omega)$ and its derived grading $\calf^{(n)}$ and filtration $\calf_{(n)}$ restrict to a grading and filtration on $\bfk\frakL_\infty$. By Eq.~(\mref{eq:rbdegfil}), they are compatible with the ones on $\frakX_\infty$. More precisely,
\begin{equation}
\frakL^{(n)}:=\calf^{(n)}\cap \frakL_\infty=\theta(\frakX^{(n)}), \quad \frakL_{(n)}:=\calf_{(n)}\cap \frakL_\infty=\theta(\frakX_{(n)}), \quad n\geq 0.
\mlabel{eq:fdegfil}
\end{equation}

Therefore by Proposition~\mref{pp:wfil} we have
\begin{prop}
Let $j_X:X\hookrightarrow \ncshal(X,\Omega)$, $x\mapsto \bullet_x, x\in X,$ be the natural embedding.
Then the triple $(\ncshal(X,\Omega),\, \dl ,\, \Bo)$ together with $j_X$ is the free \mrba of weight $\lambda_\Omega$ on $X$. Further, $\ncshal(X,\Omega)$ with $\bfk \frakL_{(n)}, n\geq 0,$ is an $\Omega$-operated filtered algebra and $\psi$ is a homomorphism of $\Omega$-operated filtered algebras.
\mlabel{prop:alge}
\end{prop}

\section{$\Omega$-cocycle Hopf algebras and free \mrbas}
\mlabel{sec:hopf}
In this section, we first derive an $\Omega$-cocycle bialgebraic structure on the free matching Rota-Baxter algebra $\ncshal(X, \Omega)$, via a construction of a suitable coproduct.
We then show that this $\Omega$-cocycle bialgebra is connected cofiltered and so a Hopf algebra.

Let $\uul: {\bfk}\rightarrow \ncshal(X, \Omega)$ be the linear map given by $1_{\bfk}\mapsto 1$.
By Proposition~\mref{prop:alge}, the triple $(\ncshal(X, \Omega),\dl,\uul)$ is an algebra.
We now define a linear map $\del:\ncshal(X, \Omega)\rightarrow\ncshal(X, \Omega)\otimes\ncshal(X, \Omega)$ by setting
\begin{equation}
\del(F):=(\psi\ot\psi)\col i (F) \quad \text{for all } ~F\in\ncshal(X, \Omega),
\mlabel{eq:lcoproduct}
\end{equation} where $i:\ncshal(X, \Omega)\to \bfk \rfs$ is the natural inclusion.
In other words, $\del$ is defined so that the diagram
$$
\xymatrix{
\ncshal(X, \Omega) \ar^(.4){\del}[rr] \ar_{i}[d] && \ncshal(X, \Omega)\ot \ncshal(X, \Omega) \\
\bfk \rfs \ar^(.4){\col}[rr] && \bfk \rfs \ot \bfk \rfs\ar_{\psi\ot \psi}[u]
}
$$
commutes.

Define $\epsl  :\ncshal(X, \Omega) \rightarrow {\bfk}$ by setting
\begin{equation}
\epsl(F) =
\left \{\begin{array}{ll}
0, &\text{ if } F\neq 1, \\
1, &\text{ if } F =1.
\end{array} \right.
\mlabel{eq:ep1}
\end{equation}

We first verify that
the coproduct $\del$ on $\ncshal(X, \Omega)$ satisfies the Hochschild 1-cocycle condition.
\begin{lemma}\mlabel{lem:3.8}
Let $F=\pl(\lbar{F})$ be in $\ncshal(X, \Omega)$. Then
\begin{equation}
\del(\pl(\lbar{F}))=\pl(\lbar{F})\ot1+(\id\ot\pl)\del(\lbar{F}).
\mlabel{eq:lTree}
\end{equation}
\end{lemma}
\begin{proof}
By the linearity, we just need to verify Eq.~(\mref{eq:lTree}) for $F\in\frak L_{\infty}$. Then
\begin{align*}
\del(\pl(\lbar{F}))
&=(\psi\ot\psi)\col i(\pl(\lbar{F}))\quad(\text{By Eq.~(\mref{eq:lcoproduct})})\\
&=(\psi\ot\psi)\col(\pl(\lbar{F}))\quad(\text{by $i$ being an inclusion map})\\
&=(\psi\ot\psi)(F\ot1+(\id\ot\pl)\col(\lbar{F}))\quad(\text{By Eq.~(\mref{eq:cocycle})})\\
&=\psi(F)\ot\psi(1)+(\psi\ot\psi\pl)\col(\lbar{F})\\
&=\psi i(F)\ot\psi(1)+(\psi\ot\psi\pl)\col i(\lbar{F})\quad(\text{by $i$ being an inclusion map})\\
&=F\ot1+(\psi\ot\psi\pl)\col i(\lbar{F})\quad(\text{by Lemma~\mref{lem:varid}})\\
&=F\ot1+(\psi\ot\pl\psi)\col i(\lbar{F})\\
&\hspace{2cm}(\text{by $\psi$ being an operated algebra homomorphism in Eq.~(\mref{eq:alghom})})\\
&=F\ot1+(\id\ot\pl)(\psi\ot\psi)\col i(\lbar{F})\\
&=F\ot1+(\id\ot\pl)\del(\lbar{F})\quad(\text{by Eq.~(\mref{eq:lcoproduct})}).
\end{align*}
This completes the proof.
\end{proof}

Next we verify the compatibility of $\del$ with $\dl$, starting with a special case.

\begin{lemma}\label{lem:Morphism0}
Let $F, F'\in\frakL_{\infty}$ with $F\dl F'=FF'$. Then
\begin{equation*}
\del(F\dl F')=\del(F)\dl\del(F').
\end{equation*}
\end{lemma}

\begin{proof}
We have
\begin{align*}
\del ( F\,\dl \, F')
&=(\psi\ot\psi)\col i(FF')\\
&=(\psi\ot\psi)\col (FF')\quad(\text{by $i$ being an inclusion map})\\
&=(\psi\ot\psi)\bigg(\col(F)\col(F')\bigg)
\quad(\text{by $\col$ being an algebra homomorphism})\\
&=\bigg((\psi\ot\psi)\col(F)\bigg)\dl\bigg( (\psi\ot\psi)\col(F')\bigg)
\quad(\text{by $\psi$ being an algebra homomorphism})\\
&=\bigg((\psi\ot\psi)\col i (F)\bigg)\dl\bigg( (\psi\ot\psi)\col i (F')\bigg)\quad(\text{by $i$ being an inclusion map})\\
&=\del (F)\,\dl \,\del ( F') \quad(\text{by Eq.~(\mref{eq:lcoproduct})}).
\end{align*}
This completes the proof.
\end{proof}
In general, we have

\begin{lemma}\mlabel{lem:compa}
Let $F, F'\in\ncshal(X, \Omega)$. Then
\begin{align}
\del(F\dl F')=\del(F)\dl\del(F') \, \text{ and }\, \epsl(F\dl F')=\epsl(F)\epsl(F').
\mlabel{eq:Morphism}
\end{align}
\end{lemma}

\begin{proof}
The second equation follows from the definition of $\epsl$ in Eq.~(\mref{eq:ep1}).

For the first equation, by the linearity, we just need to consider the case when $F,F'\in \frak L_{\infty}$.
We apply the induction on the sum of depths $s:=\dep(F)+\dep(F')\geq 0$. For the initial step of $s=0$, we have $\dep(F) =\dep(F') = 0$ and so $F\dl F' =FF'$. Then Eq.~(\mref{eq:Morphism}) follows from Lemma~\mref{lem:Morphism0}.

Let $t\geq 0$. Assume that Eq.~(\mref{eq:Morphism}) holds for $s=t$ and consider the case of $s=t+1$. In this case, we first consider the case when $\bre(F)=\bre(F')=1$.
If $F\dl F' = FF'$, then Eq.~(\mref{eq:Morphism}) follows from Lemma \mref{lem:Morphism0}. If $F\dl F' \neq FF'$, then we have
$F=\pla(\lbar{F})\,\text{ and }\, F'=\plb(\lbar{F}')$ for some $\alpha, \beta\in \Omega$ and $\lbar{F},\lbar{F}'\in\frak L_{\infty}.$
Write
\begin{equation}
\del(\lbar{F}):=\sum_{(\lbar{F})}\lbar{F}_{(1)}\ot\lbar{F}_{(2)}\,\text{ and }\,\del(\lbar{F}'):=\sum_{(\lbar{F}')}\lbar{F}'_{(1)}\ot\lbar{F}'_{(2)}.
\mlabel{eq:coff}
\end{equation}
Then
\begin{align*}
&\del(F\dl F')\\
=&\ \del(\pla(\lbar{F})\dl \plb(\lbar{F}')) \\
=&\ \del\Big(\pla(\lbar{F}\dl \plb(\lbar{F}'))+\plb(\pla(\lbar{F})\dl \lbar{F}')+\lambda_\beta \pla(\lbar{F}\dl \lbar{F}') \Big)\quad(\text{by Proposition \mref{prop:alge}}) \\
=&\ \del\pla(\lbar{F}\dl \plb(\lbar{F}'))+\del\plb(\pla(\lbar{F})\dl \lbar{F}')+\lambda_\beta \del\pla(\lbar{F}\dl \lbar{F}')  \\
=&\ \pla(\lbar{F}\dl \plb(\lbar{F}'))\ot \etree+ (\id \ot \pla)\del\Big(\lbar{F}\dl \plb(\lbar{F}')\Big)
+\plb(\pla(\lbar{F})\dl \lbar{F}')\ot \etree \\
&\ +(\id \ot \plb)\del \Big(\pla(\lbar{F})\dl \lbar{F}'\Big)+\lambda_\beta \pla(\lbar{F}\dl \lbar{F}')\ot \etree+\lambda_\beta(\id \ot \pla)\del(\lbar{F}\dl \lbar{F}')\quad(\text{by Eq.~(\mref{eq:lTree})}) \\
=&\ \pla(\lbar{F})\dl \plb(\lbar{F}')\ot \etree+ (\id \ot \pla)\del\Big(\lbar{F}\dl \plb(\lbar{F}')\Big)
 +(\id \ot \plb)\del \Big(\pla(\lbar{F})\dl \lbar{F}'\Big)\\
&\ +\lambda_\beta(\id \ot \pla)\del(\lbar{F}\dl \lbar{F}')\quad(\text{by Proposition \mref{prop:alge}}) \\
=&\ (F\dl F')\ot 1+(\id\ot\pla) \Big(\del(\lbar{F})\dl \del\plb(\lbar{F}')\Big) + (\id\ot\plb)\Big(\del\pla(\lbar{F})\dl \del(\lbar{F}')\Big)\\
&\ +\lambda_\beta (\id\ot\pla) \Big(\del(\lbar{F})\dl \del(\lbar{F}')\Big) \quad
(\text{by the induction hypothesis on}~s)\\
=&\ (F\dl F')\ot 1+\sum_{(\lbar{F})}(\lbar{F}_{(1)}\dl F')\ot\pla(\lbar{F}_{(2)})+\sum_{(\lbar{F}')}(F\dl\lbar{F}'_{(1)})\ot\plb(\lbar{F}'_{(2)})\\
&\ +\sum_{(\lbar{F})}\sum_{(\lbar{F}')}(\lbar{F}_{(1)}\dl\lbar{F}'_{(1)})
\ot(\pla(\lbar{F}_{(2)})\dl\plb(\lbar{F}'_{(2)}))\quad(\text{by Eqs.~(\mref{eq:lTree}) and ~(\mref{eq:coff})})\\
=&\ \bigg(F\ot 1+\sum_{(\lbar{F})}\lbar{F}_{(1)}\ot\pla(\lbar{F}_{(2)})\bigg)\dl\bigg(F'\ot 1+\sum_{(\lbar{F}')}\lbar{F}'_{(1)}\ot\plb(\lbar{F}'_{(2)})\bigg)\\
=&\ \bigg(F\ot 1+(\id\ot\pla)\del(\lbar{F})\bigg)\dl\bigg(F'\ot 1+(\id\ot\plb)\del(\lbar{F}')\bigg)\\
=&\ \del(F)\dl \del(F').
\end{align*}
We next consider the general case when $\bre(F)> 1$ or $\bre(F')> 1$ (or both). We just assume that both hold since the other cases are similar. Then we can write
$F=F_{0}T_{1}$ or  $F'=T'_{0}F'_{1}$,
for some matching Rota-Baxter forests $F_{0}, F'_{1}$ and some matching Rota-Baxter trees $T_{1},T'_{0}$.
Since
$$F_{0} (T_{1}\dl T'_{0}) F'_{1} =F_{0}\dl (T_{1}\dl T'_{0})\dl F'_{1},$$ we have
\begin{align*}
\del(F\dl F')
&=\del( F_{0} (T_{1}\dl T'_{0}) F'_{1})\\
&=\del(F_{0})\dl \del(T_1\dl T'_0) \dl \del(F'_{1}) \quad (\text{by Lemma~\mref{lem:Morphism0}})\\
&=\del(F_{0})\dl \del(T_{1})\dl \del(T'_{0})\dl\del( F'_{1})\quad(\text{by the case when }\bre(F)=\bre(F')=1)\\
&=\big(\del(F_{0})\dl \del(T_{1})\big)\dl \big(\del(T'_{0})\dl\del( F'_{1})\big)\\
&=\del(F_{0}\dl T_{1})\dl\del(T'_{0}\dl F'_{1})\quad(\text{by Lemma~\mref{lem:Morphism0}})\\
&=\del(F_{0} T_{1})\dl\del(T'_{0}F'_{1})\\
&=\del(F)\dl\del(F').
\end{align*}
This completes the proof.
\end{proof}

\begin{theorem}
The sextuple $(\ncshal(X, \Omega),\,
\dl, \, \uul, \, \del,\, \epsl,\, \Bo)$ is an $\Omega$-cocycle bialgebra.
\mlabel{them:cocbia}
\end{theorem}

\begin{proof}
By Lemmas~\mref{lem:3.8} and \mref{lem:compa}, we only need to verify the coassiciativity of $\del$ and the  counicity of $\epsl$.
For the coassociativity of $\del$, following the idea of the proof of Theorem~\mref{thm:propm}, we just need to show that the set
\begin{align*}
\mathscr{C}:= \{F \in \ncshal(X, \Omega)\mid \, (\del \ot \id)\del(F)= (\id \ot \del) \del(F) \}
\end{align*}
is an $\Omega$-operated subalgebra of $\ncshal(X,\Omega)$.

Note that $1$ is in $\mathscr{C}$. By Lemma~\mref{lem:compa}, $\del$ is an algebra homomorphism. Then $(\del \ot \id) \del$ and $(\id \ot \del) \del$ are also algebra homomorphisms, implying that $\mathscr{C}$ is a subalgebra of $\ncshal(X, \Omega)$. For any $x \in X$, we have
\begin{align*}
(\del \ot \id) \del(\bullet_{x})&\ =(\del \ot \id) (1 \ot \bullet_{x}+ \bullet_{x} \ot 1)\quad \text{(by Eq.~(\mref{eq:lcoproduct}))}\\
&\ =1 \ot 1 \ot \bullet_{x}+ 1 \ot \bullet_{x} \ot 1 + \bullet_{x} \ot 1 \ot 1\\
&=(\id \ot \del)\del(\bullet_{x}).
\end{align*}
Thus $\bullet_{x} \in \mathscr{C}$. For any $F \in \mathscr{C}$ and $\omega \in \Omega$, we have
\begin{align*}
& (\del \ot \id)\del(B_{\omega}^{+}(F))\\
=& (\del \ot \id)(B_{\omega}^{+}(F) \ot 1+ (\id \ot B_{\omega}^{+}) \del(F)) \quad \text{(by Lemma~\ref{lem:3.8})}\\
=& \del(B_{\omega}^{+}(F)) \ot 1+ (\del \ot B_{\omega}^{+}) \del(F)\\
=& B_{\omega}^{+}(F) \ot 1 \ot 1+ (\id \ot B_{\omega}^{+}) \del(F) \ot 1+(\id \ot \id \ot B_{\omega}^{+})(\del \ot \id)\del(F) \quad \text{(by Lemma~\ref{lem:3.8})}\\
=& B_{\omega}^{+}(F) \ot 1 \ot 1+(\id \ot B_{\omega}^{+})\del(F) \ot 1+ (\id \ot \id \ot B_{\omega}^{+})(\id \ot \del)\del(F)\quad \text{(by $F\in \mathscr{C}$)}\\
=& B_{\omega}^{+}(F) \ot 1 \ot 1+\Big((\id \ot B_{\omega}^{+}) \ot 1+ (\id \ot \id \ot B_{\omega}^{+})(\id \ot \del)\Big)\del(F)\\
=& B_{\omega}^{+}(F) \ot 1 \ot 1+ (\id \ot \del B_{\omega}^{+})\del(F)\quad \text{(by Lemma~\ref{lem:3.8})}\\
=& (\id \ot \del)(B_{\omega}^{+}(F) \ot 1+(\id \ot B_{\omega}^{+}) \del(F)) \\
=& (\id \ot \del) \del(B_{\omega}^{+}(F))\quad \text{(by Lemma~\ref{lem:3.8})},
\end{align*}
which means that $B_{\omega}^{+}(F) \in \mathscr{C}$. Thus $\mathscr{C}$ is stable under $B_{\omega}^{+}$ for any $\omega \in \Omega$. In summary, $\mathscr{C}=\ncshal(X, \Omega)$ and so $\del$ is coassociative.

For the counicity, we similarly consider
\begin{align*}
\mathscr{D}:= \left\{F \in \ncshal(X, \Omega) \mid \, (\epsl \ot \id) \del(F)=\beta_{l}(F)\, \text{ and }\, ( \id \ot \epsl ) \del(F)=\beta_{r}(F)  \right\},
\end{align*}
adapting notations from~\mcite{Gub}. Note that $1 \in \mathscr{D}$. By Lemma~\mref{lem:compa}, $\del$ and $\epsl$ are algebra homomorphisms, so $\mathscr{D}$ is a subalgebra of $\ncshal(X, \Omega)$. For any $x \in X$, by Eq.~(\mref{eq:ep1}), we have
\begin{align*}
(\epsl \ot \id) \del(\bullet_{x})&\ = (\epsl \ot \id)(1 \ot \bullet_{x}+ \bullet_{x} \ot 1) =1_\bfk \ot \bullet_{x}=\beta_{l}(\bullet_x)
\end{align*}
and
\begin{align*}
(\id \ot \epsl) \del(\bullet_{x})&\ = (\id \ot \epsl)(1 \ot \bullet_{x}+ \bullet_{x} \ot 1) =\bullet_{x} \ot 1_\bfk =\beta_{r}(\bullet_x).
\end{align*}
Hence $\bullet_{x} \in \mathscr{D}$. For any $F \in \mathscr{D}$ and $\omega \in \Omega$, we have
\begin{align*}
&\ (\epsl \ot \id) \del(B_{\omega}^{+}(F))\\
=&\  (\epsl \ot \id)(B_{\omega}^{+}(F) \ot 1+ (\id \ot B_{\omega}^{+})\del(F)) \quad \text{(by Lemma~(\ref{lem:3.8}))}\\
=&\  \epsl(B_{\omega}^{+}(F)) \ot 1+ (\epsl \ot \id)(\id \ot B_{\omega}^{+}) \del(F)\\
=&\  \epsl(B_{\omega}^{+}(F)) \ot 1+ (\id \ot B_{\omega}^{+})(\epsl \ot \id) \del(F)\\
=&\ 0+ (\id \ot B_{\omega}^{+}) \beta_{l}(F) \quad \text{(by Eq.~(\mref{eq:ep1}) and $F \in \mathscr{D}$ )}\\
=&\ 1_\bfk \ot B_{\omega}^{+}(F)
= \beta_{l}(B_{\omega}^{+}(F))
\end{align*}
and
\begin{align*}
&\ (\id \ot\epsl) \del(B_{\omega}^{+}(F))\\
=&\  (\id \ot\epsl)(B_{\omega}^{+}(F) \ot 1+ (\id \ot B_{\omega}^{+})\del(F)) \quad \text{(by Lemma~(\ref{lem:3.8}))}\\
=&\  B_{\omega}^{+}(F) \ot 1_\bfk+ (\id \ot\epsl)(\id \ot B_{\omega}^{+}) \del(F)\\
=&\ B_{\omega}^{+}(F) \ot 1_\bfk + \sum_{(F)} F_{(1)}\ot \epsl (B_{\omega}^{+}(F_{(2)}))\\
=&\ B_{\omega}^{+}(F) \ot 1_\bfk+ 0 \quad \text{(by Eq.~(\mref{eq:ep1}) and $F \in \mathscr{D}$ )}\\
=&\ \beta_{r}(B_{\omega}^{+}(F)).
\end{align*}
Thus $B_{\omega}^{+}(F) \in \mathscr{D}$ and $\mathscr{D}$ is stable under $B_{\omega}^{+}$ for any $\omega \in \Omega$. Consequently, $\mathscr{D}=\ncshal(X, \Omega)$ and $\epsl$ is a  counit. This completes the proof.
\end{proof}

We now introduce the connectedness condition on coalgebras~\mcite{GG19}.

\begin{defn} \mlabel{def:fil}
\begin{enumerate}
\item  A coalgebra $(C, \Delta, \vep)$ is called {\bf coaugmented} if there is a linear map $u: \bfk \rightarrow C$, called the {\bf coaugmentation}, such that $\vep u=\id_\bfk$.
\item
A coaugmented coalgebra $(C,u,\Delta,\vep)$ is called {\bf cofiltered} if there is an exhaustive increasing filtration $\{C_{(n)}\}_{n\geq 0}$ of $H$ such that
\begin{equation}
\im\,u\subseteq C_{(n)}, \quad \Delta(C_{(n)})\subseteq\sum\limits^{}_{p+q=n}C_{(p)}\otimes C_{(q)}, \quad n\geq0,  p,q\geq0
\mlabel{eq:cofil}
\end{equation}
Elements in $C_{(n)}\setminus C_{(n-1)}$ are said to have degree $n$. $C$ is called {\bf connected (filtered)} if in addition $C^{0}=\mathrm{im}\, u~(={\bfk})$.
\mlabel{def:b}
\end{enumerate}
\end{defn}

By the coaugmented condition, we have $C=\im\, u \oplus \ker \vep$. Then from $\im\,u \subseteq C_{(n)}$ and modularity, we have $C_{(n)}=\im\, u\oplus (C_{(n)} \cap \ker \vep)$, as originally stated in~\mcite{GG19}.

We quote the following condition~\cite[Theorem~3.4]{GG19} for Hopf algebras. See also~\mcite{LR06,Man01,Mon93}.

\begin{lemma}\label{lem:filterhopf}
Let $H=(H,m,u,\Delta,\vep)$ be a bialgebra such that $(H, \Delta,\vep)$ is a connected cofiltered coaugumented coalgebra. Then $H$ is a Hopf algebra.
\end{lemma}

Finally, we arrive at our main result on Hopf algebraic structure on free \mrbas.
\begin{theorem}
The sextuple $(\ncshal(X, \Omega),\, \dl,\, \uul,\, \del,\, \epsl, \, \Bo)$ is an $\Omega$-cocycle Hopf algebra.
\mlabel{thm:cohopf}
\end{theorem}
\begin{proof}
By Theorem \mref{them:cocbia}, The quintuple $(\ncshal(X, \Omega),\, \dl,\, \uul,\, \del,\, \epsl, \, \Bo)$ is an $\Omega$-cocycle bialgebra. In particular, $(\ncshal(X, \Omega),\, \uul,\, \del,\, \epsl)$ is a coaugumented coalgebra.
To apply Lemma~\mref{lem:filterhopf}, we just need to check that $(\ncshal(X, \Omega),\, \uul,\, \del,\, \epsl)$ is connected and cofiltered, with respect to the filtration $\bfk \calf_{(n)}, n\geq 0$. The connectedness is clear since $\calf_{(0)}=1$. Further by Proposition~\mref{prop:alge} and Eq.~\eqref{eq:lcoproduct}, we have
\begin{align*}
\del(\frakL_n) = (\psi \ot \psi) (\del (\frakL_n)) \subseteq (\psi \ot \psi) \Big (\sum_{p+q=n} (\bfk \calf_{(p)})\ot (\bfk \calf_{(q)})\Big)
\subseteq \sum_{p+q=n} (\bfk\frakL_{(p)}) \ot (\bfk \frakL_{(q)}).
\end{align*}
Now the conclusion follows from Lemma~\mref{lem:filterhopf}.
\end{proof}

\noindent {\bf Acknowledgments}: This work was supported by the National Natural Science Foundation of China (Grant No.\@ 11771190, 11771191). 

\end{document}